\documentclass[final,hidelinks,onefignum,onetabnum]{siamart250211}

\usepackage{lipsum}
\usepackage{amsfonts}
\usepackage{graphicx}
\usepackage{epstopdf}
\usepackage{algorithmic}
\ifpdf
\DeclareGraphicsExtensions{.eps,.pdf,.png,.jpg}
\else
\DeclareGraphicsExtensions{.eps}
\fi

\newsiamremark{remark}{Remark}
\newsiamremark{hypothesis}{Hypothesis}
\crefname{hypothesis}{Hypothesis}{Hypotheses}
\newsiamthm{claim}{Claim}
\newsiamremark{fact}{Fact}
\crefname{fact}{Fact}{Facts}

\headers{Error estimates for NLSE with singular potential}{W. Bao and C. Wang}

\title{Error estimates of an exponential wave integrator for the nonlinear Schr\"odinger equation with singular potential\thanks{Submitted to the editors DATE.
\funding{This work was partially supported by the Ministry of Education of Singapore under its
	AcRF Tier 1 funding grant A-8003584-00-00 (W. Bao) and Tier 2 funding grant T2EP20222-0001 (A-8001562-00-00) (C. Wang).}}}

\author{Weizhu Bao\thanks{Department of Mathematics, National University of Singapore, Singapore 119076
		(\email{matbaowz@nus.edu.sg}, \email{e0546091@u.nus.edu})} 
	\and Chushan Wang\footnotemark[2]}

\usepackage{amsopn,bm,enumerate,amssymb}

\newcommand{\R}{\mathbb{R}}
\newcommand{\C}{\mathbb{C}}
\newcommand{\Z}{\mathbb{Z}}
\newcommand{\Stau}{S_\tau}
\newcommand{\vxi}{\bm{\xi}}
\newcommand{\vx}{\bm{x}}

\newcommand{\rmd}{\mathrm{d}}
\newcommand{\vep}{\varepsilon}
\newcommand{\vphi}{\varphi}

\DeclareMathOperator*{\esssup}{ess\,sup}

\begin{document}
	
	\maketitle
	
	\begin{abstract}
We analyze a first-order exponential wave integrator (EWI) for the nonlinear Schr\"odinger equation (NLSE) with a singular potential that is locally in $L^2$, which might be locally unbounded. A typical example is the inverse power potential such as the Coulomb potential, which is the most fundamental potential in quantum physics and chemistry. We prove that, under the assumption of $L^2$-potential and $H^2$-initial data, the $L^2$-norm convergence of the EWI is, roughly, first-order in one dimension (1D) and two dimensions (2D), and $\frac{3}{4}$-order in three dimensions (3D). In addition, under a stronger integrability assumption of $L^p$-potential for some $p>2$ in 3D, the $L^2$-norm convergence increases to almost ${\frac{3}{4}} + 3(\frac{1}{2} - \frac{1}{p})$ order if $p \leq \frac{12}{5}$ and becomes first-order if $p > \frac{12}{5}$. In particular, our results show, to the best of our knowledge for the first time, that first-order $L^2$-norm convergence can be achieved when solving the NLSE with the Coulomb potential in 3D. The key advancements are the use of discrete (in time) Strichartz estimates, which allow us to handle the loss of integrability due to the singular potential that does not belong to $L^\infty$, and the more favorable local truncation error of the EWI, which requires no (spatial) smoothness of the potential. Extensive numerical results in 1D, 2D, and 3D are reported to confirm our error estimates and to show the sharpness of our assumptions on the regularity of the singular potentials.  
	\end{abstract}
	
	
	
	\section{Introduction}
	We consider the following Cauchy problem of the nonlinear Schr\"odinger equation (NLSE) of the form
	\begin{equation}\label{NLSE}
		\left\{
		\begin{aligned}
			&i \partial_t \psi(\vx, t) = -\Delta \psi(\vx, t) + V(\vx) \psi(\vx, t) + \beta|\psi(\vx, t)|^{2\sigma} \psi(\vx, t), \quad \vx \in \R^d, \\
			&\psi(\vx, 0) = \psi_0(\vx), \quad \vx \in \R^d, 
		\end{aligned}
		\right.
	\end{equation}
	where $t\geq0$ is time, $\vx = (x_1, \cdots, x_d)^T$ is the space variable, $d \in \{1, 2, 3\}$ represents the spatial dimension, $\psi = \psi(\vx, t) \in \C$ is the wave function, and $\beta \in \R$, $\sigma \in \R^+$ are given parameters characterizing the nonlinearity. Here, $V := V(\vx)$ is a real-valued external potential and we assume that \cite{cazenave2003}
	\begin{equation}
		V \in L^2(\R^d) + L^\infty(\R^d) := \{ a u + b v\, |\, u \in L^2(\R^d), \ v \in L^\infty(\R^d), \  a, b \in \R \}, 
	\end{equation}
	to allow for possible singularities.

	The NLSE \cref{NLSE} has diverse applications in quantum mechanics, quantum physics and chemistry, laser beam propagation, plasma and particle physics, deep water waves, etc \cite{book_NLSE_Sulem,review_2013,book_NLSE_Gadi}. Especially, when $\sigma=1$, it is widely used in the modeling and simulation of Bose-Einstein condensates (BEC), where it is also known as the Gross-Pitaevskii equation (GPE) \cite{BEC_book,ESY}. In these applications, the potential $V$ may take different forms due to the different theoretical and experimental settings. In particular, in many cases, the potential $V$ is of low regularity and could be singular \cite{singular_poten_PR,singular_poten_RevModPhys,singular_poten_1964}. The most important and fundamental example is the Coulomb potential with $V(\vx) = -1/|\vx|$, which is the quantum mechanical description of the Coulomb force between charged particles \cite{SE}. Other singular potentials include the inverse square potential with $V(\vx) = -1/|\vx|^2$, the Yukawa potential (or screened Coulomb potential) with $V(\vx) = - e^{-|\vx|}/|\vx|$, the Dirac delta potential, etc \cite{inverse_square,Yukawa,delta}. Another important class of low-regularity or singular potential arises in the study of wave propagation in disordered media and the phenomenon of localization~\cite{Anderson,disorder_science,mbl}, where the disorder potential may belong to spaces such as $L^\infty$, $L^2$, or even the negative Sobolev spaces $H^{-\alpha}$ for some $\alpha > 0$, depending on the strength and nature of the disorder.

	For the (local) well-posedness of the NLSE \cref{NLSE} with the singular potential of $L^p$-type, this has been well studied in the PDE literature, and we briefly summarize the relevant results below (see, e.g., the book \cite{cazenave2003}). When $V \in L^2(\R^d) + L^\infty(\R^d)$ and $\sigma > 0$ in \cref{NLSE}, the $H^2$ well-posedness is obtained, which is also recalled later. In fact, such $H^2$ well-posedness under $L^2$-potential should be sharp, according to the ill-posedness results in \cite{zhao2024}. For $H^1$-solution, the local well-posedness is obtained for $V \in L^p(\R^d) + L^\infty(\R^d)$ with $p \geq 1$ and $p>d/2$, and $0<\sigma<2/(d-2)$ ($0<\sigma<\infty$ if $d=1, 2$). Moreover, for $V \in L^p(\R^d) + L^\infty(\R^d)$ with $p \geq 1$ and $p>d/2$, and $0 < \sigma < 2/d$, the NLSE \cref{NLSE} is well-posed in $L^2$. 
	
	Along the numerical side, many accurate and efficient numerical methods have been proposed in the literature to solve the NLSE, including the finite difference time domain (FDTD) method \cite{FD,bao2013,Ant,henning2017}, the time-splitting method \cite{BBD,bao2003,lubich2008,splitting_low_reg,su2022,bao2023_semi_smooth,splittinglowregfull}, the exponential wave integrator (EWI) \cite{ExpInt,SymEWI,bao2023_EWI}, and the low regularity integrator \cite{LRI,tree1,LRI2024,LRI,tree2}. However, most of these methods are proposed and analyzed under the assumption of sufficiently smooth potential (usually at least $H^1 \cap L^\infty$ or $H^2$). Very recently, efforts have been made to establish error estimates under low-regularity assumptions on the potential. Nevertheless, the potential is still required to be (locally) bounded. For results concerning error estimates under low regularity $L^\infty$-potential, we refer to \cite{henning2017} for the FDTD methods, \cite{bao2023_EWI,bao2023_sEWI} for the EWI, and \cite{bao2023_improved,bao2024} for the time-splitting methods. Despite these advances, the assumption of $L^\infty$-potential excludes certain important singular potentials--most notably the Coulomb potential, which is the most important one in applications. 
	
	In fact, despite its wide range of applications, very few results are available concerning the error analysis of numerical methods for the NLSE with the singular potential of $L^p$-type. Recently, the splitting method has been analyzed for the linear Schrödinger equation with such singular potentials in \cite{PRR,singular2024}. The results in \cite{PRR} are valid only for very specific initial data—namely, eigenstates of the corresponding linear Schrödinger operator—while the estimates in \cite{singular2024} apply to general initial data with certain regularity. However, the convergence rate obtained in both works is rather low; for example, only $\frac{1}{4}$-order convergence in $L^2$-norm is achieved for Coulomb potential in 3D. In fact, such a low convergence rate is not due to the proof techniques but confirmed by the numerical experiments. Moreover, their results are limited to the linear case and cannot be applied when there is nonlinearity in \cref{NLSE}. In another recent development \cite{zhao2024}, a novel low regularity integrator is proposed and the error estimate is established for the 1D cubic NLSE with singular potentials. The error estimate in \cite{zhao2024} is valid even for certain distributional potential; however, the convergence rate for $L^2$-potential is proved to be $\frac{3}{4}$-order in $L^2$-norm and the results are restricted to 1D. We would also like to mention a completely different approach proposed in \cite{shao2023} by turning to the Wigner function dynamics such that the singular potential is converted to the Wigner kernel with weak or even no singularity. In light of the above, the rigorous numerical analysis for the NLSE with the singular potential of $L^p$-type remains largely open. In particular, to the best of our knowledge, there seem to be no results concerning the standard EWI under singular potentials, despite its popularity in solving the smooth NLSE. On the other hand, the analysis in \cite{bao2023_EWI,bao2023_sEWI}, which shows that the EWI performs favorably under low regularity $L^\infty$-potential, suggests that it holds significant promise to handle more singular potentials. 
	
	In this work, we analyze a standard first-order EWI, also known as the exponential Euler scheme in the literature, for the NLSE \cref{NLSE} with the singular potential in $L^2 + L^\infty$. Our main results are as follows: 
	\begin{enumerate}[(i)]
		\item Under the weakest assumption of $L^2+L^\infty$-potential, and $H^2$-initial data, the $L^2$-norm convergence of EWI is, roughly, first-order in 1D and 2D, and $\frac{3}{4}$-order in 3D (\cref{thm:L2poten}). 
		
		\item When assuming stronger integrability of $L^p+L^\infty$-potential with $p>2$ in 3D, we obtain higher convergence order of $\frac{3}{4} + 3(\frac{1}{2} - \frac{1}{p})$ when $p \leq \frac{12}{5}$ and first-order convergence when $p>\frac{12}{5}$ (\cref{thm:optimal}). 
	\end{enumerate}
	As an immediate corollary, we obtain error estimates for the singular potential with inverse power type singularities in \cref{cor:inverse_power}, indicating that the EWI is optimally first-order convergent in $L^2$-norm for the NLSE with Coulomb potential in 3D (\cref{cor:coulomb}). To the best of our knowledge, this is the first time that first-order $L^2$-norm convergence is proved for the NLSE with 3D Coulomb potential. Moreover, the numerical results presented in \cref{sec:4} indicate that our assumption of $L^2+L^\infty$-potential for first-order $L^2$-norm convergence of the EWI (at least in 1D) is optimally weak. 
	
	{
	In the following, we briefly explain the idea and novelty of our proof. There are two main difficulties in establishing error estimates under singular potentials of $L^p$-type. Firstly, the singular potential introduces loss of integrability, and thus destroys the usual $L^2$-stability of the numerical scheme. To be precise, for $V \in L^p$ with $2 \leq p < \infty$, the operator $\phi \rightarrow V \phi$ is bounded from $L^{\frac{2p}{p-2}}$ to $L^2$ instead of from $L^2$ to $L^2$. Here, $\frac{2p}{p-2} > 2$ for $2 \leq p < \infty$ shows the loss of integrability, making the EWI that we consider (indeed, most of exponential-type integrators) no longer stable in $L^2$. Although such loss of integrability and the resulting loss of $L^2$-stability introduces significant difficulties from the numerical perspective, it has been successfully overcome at the continuous level by the Strichartz estimates, which features a gain of the spatial integrability at the cost of some temporal integrability. Motivated by this, we apply the discrete Strichartz estimates (i.e. \cref{lem:dS1,lem:dS2,lem:dSc}), and establish the stability of the numerical scheme in some space-time norm $l^q([0, T]; L^r)$ (defined in \cref{eq:norm_def}) with $T>0$ being the final time in \cref{prop:Z}. Note that such stability is held globally by the whole sequence of numerical solutions, which is different from the more standard step-wise stability. We remark here that the discrete Strichartz estimates, first proposed in \cite{ignat2011}, have recently attracted more attention in obtaining error estimates of the NLSE under low regularity initial data \cite{LRI_error_1,choi2021,choi2025,su2022}, where they are used to control the nonlinearity. However, to our best knowledge, this is the first time that these estimates are used to establish error estimates under singular potentials. 
	
	The second difficulty is that the singular potential leads to a significant loss of the (local) convergence order. This phenomenon is typical for time-splitting methods \cite{singular2024,PRR,bao2023_semi_smooth}: based on the standard commutator estimates \cite{jahnke2000,lubich2008}, each convergence order requires to apply two additional spatial derivatives to the potential (equivalently, bound terms of the form $\Delta(V \psi)$). This explains the extremely low convergence orders proved and observed in \cite{singular2024,PRR}. On the other hand, our recent results in \cite{bao2023_EWI,bao2023_sEWI} indicate that the EWI may partially overcome such convergence order reduction as each convergence order requires controlling the first-order time derivative of the potential term (i.e. $\partial_t (V \psi) = V \partial_t \psi$). When $V \in L^\infty$, this is sufficient to recover the optimal local temporal convergence order in $L^2$-norm, since $\| V \partial_t \psi \|_{L^2} \leq \| V \|_{L^\infty} \| \partial_t \psi \|_{L^2}$ \cite{bao2023_EWI,bao2023_sEWI}. However, for singular $V$, even though the time derivative shifts the regularity requirement from $V$ to $\psi$, the singularity of $V$ still destroys the $L^2$-integrability of $\partial_t \psi$, resulting in the loss of local convergence orders. To address this issue, we again invoke discrete Strichartz estimates in a global-in-time manner across the time steps, and combine them with the space-time estimate of $\partial_t \psi$ from the known regularity theory \cite{cazenave2003}. Somewhat surprisingly, this yields an optimal first-order temporal convergence in $L^2$-norm even for merely $L^2$-potentials (see \cref{prop:E2}). As mentioned above, such optimal temporal convergence relies crucially on the fact that $\partial_t \psi$ admits the same space time estimates as $\psi$ itself. To our knowledge, an error analysis with the optimal temporal convergence depends on the space-time estimate of the time derivative of the exact solution has not been reported in the literature.}

	After addressing the two main difficulties, we can establish the error estimates. As a result of the sufficiently high-order convergence we obtained for the EWI and the $H^2$-regularity of the exact solution, the nonlinearity can be easily controlled by the combination of inverse estimates and the mathematical induction, which are standard. We remark here that due to the use of discrete Strichartz estimates, we have to work under the entire space $\R^d$ and thus the numerical scheme is essentially a semidiscretization. It will be our future work to extend the results to the full discretization on some bounded domain with, for example, periodic boundary conditions \cite{LRI_error_2,splittinglowregfull,splitting_low_reg}. 

	{The techniques developed in this paper for the error estimates in the presence of singular potentials are general and applicable to a broad class of numerical methods and dispersive equations. In particular, they extend directly to (filtered) time-splitting methods based on the analysis in \cite{bao2023_improved}, yielding error estimates with improved convergence orders compared to the non-filtered schemes considered in \cite{PRR,singular2024}. Moreover, the approach can be pushed to more singular potentials than $L^2$-potentials as long as the regularity of the exact solution and the space-time estimates of its time derivative are available.}
	
	The rest of the paper is organized as follows. In \cref{sec:2}, we present the EWI and state our main results. In \cref{sec:3}, we introduce some technical estimates including the discrete Strichartz estimates. \Cref{sec:4} is devoted to the error estimates of the EWI under singular potentials. Numerical results are reported in \cref{sec:5}. Finally, some conclusions are drawn in \cref{sec:6}. Throughout the paper, we adopt standard notations of Sobolev spaces on $\R^d$ with corresponding norms. For $a \in \R$, $a^-$ denotes $a-\vep$ for any small $\vep>0$. For any $p \in [1, \infty]$, $p' \in [1, \infty]$ is the H\"older conjugate of $p$ such that $\frac{1}{p}+\frac{1}{p'} =1$. We denote by $C$ a generic positive constant independent of the time step size $\tau$, and by $C(\alpha)$ a generic positive constant depending on the parameter $\alpha$. The notation $A \lesssim B$ is used to represent that there exists a generic constant $C>0$, such that $|A| \leq CB$. 
	
	\section{The exponential wave integrator and main results}\label{sec:2}
	In this section, we present the exponential wave integrator and the main error estimate results. 
	\subsection{The exponential wave integrator (EWI)}
	Choose $\tau > 0$ to be the time step size, and denote the time steps as $t_n = n \tau$ for $n = 0, 1, \cdots$. Similar to \cite{ignat2011,LRI_error_1,choi2021,choi2025}, introduce the following filter function $\Pi_\tau$ as 
	\begin{equation}\label{eq:Pitau_def}
		(\Pi_\tau \phi)(\vx) = \int_{\R^d} \chi(\tau^{\frac{1}{2}} {\vxi}) \widehat{\phi}(\vxi) e^{i \vx \cdot \vxi} \rmd \vxi, \quad \vx \in \R^d, 
	\end{equation}
	where $\chi \in C^\infty(\R^d)$ satisfies $\chi(\vx) \equiv 1$ for $\vx \in B({\bf 0}, 1):=\{\vx\in\R^d\,|\, |\vx|\le1\}$ and is supported on $B({\bf 0}, 2) $. 
	
	Let $\psi_\tau^n$ be the approximation to $\psi(\cdot, t_n)$ for $n \geq 0$. Then the exponential wave integrator (EWI) reads \cite{ExpInt,bao2023_EWI}
	\begin{equation}\label{eq:EWI_scheme}
		\begin{aligned}
			&\psi_\tau^{n+1} = e^{i \tau \Delta} \psi^n_\tau - i \tau \vphi_1(i \tau \Delta) \Pi_\tau (V \psi_\tau^n + \beta |\psi^n_\tau|^{2\sigma} \psi_\tau^n), \quad n \geq 0, \\
			&\psi^0_\tau = \Pi_\tau \psi_0, 
		\end{aligned} 
	\end{equation}
	where $\vphi_1(z) = \frac{e^z - 1}{z}$ for $z \in \C$, and $\vphi_1(i \tau \Delta)$ is the Fourier multiplier with symbol $\vphi_1(- i \tau |\vxi|^2)$. Compared to the standard first-order EWI in \cite{ExpInt,bao2023_EWI}, we have added a filter function $\Pi_\tau$ in the front of the potential and nonlinearity in \cref{eq:EWI_scheme}. This will allow us to use the discrete Strichartz estimates to be presented in the next section. {In fact, since the time is discrete, the filter function is necessary for the discrete Strichartz estimates to hold. We refer to \cite{ignat2011,choi2021,LRI_error_1} for more details. Similar filter functions are also introduced in \cite{LRI_error_2,splitting_low_reg} to handle low regularity initial data using the discrete Bourgain spaces.} 
	
	\subsection{Main results}
	In the following, we present our main error estimate results for the EWI \cref{eq:EWI_scheme}. 
	
	We first recall the $H^2$ well-posedness of the NLSE \cref{NLSE} \cite[Theorem 4.8.1]{cazenave2003}: Under the assumptions that $V \in L^2(\R^d) + L^\infty(\R^d)$ and $\sigma > 0$, if the initial data $\psi_0 \in H^2(\R^d)$, there exists a maximal existing time $0<T_\text{max} \leq \infty$ such that 
	\begin{equation}\label{eq:H2-solution}
		\psi \in C([0, T_\text{max}); H^2(\R^d)) \cap C^1([0, T_\text{max}); L^2(\R^d)). 
	\end{equation}
	{The above $H^2$ well-posedness is also sharp in the sense that the NLSE \cref{NLSE} is ill-posed in any $H^\gamma$-spaces with $\gamma > 2$ when $V \in L^p(\R^d) + L^\infty(\R^d)$ with $2 \leq p \leq \infty$ \cite{wu2025_singular1d,wu2025_singularhd}. Hence, even if the initial data has better regularity than $H^2$, the solution cannot have better-than-$H^2$ regularity in general. }  
	
	Let $0<T<T_{\text{max}}$ be a fixed final time. We say $p \in \R$ satisfies \cref{eq:pd_def}, if
	\begin{equation}\label{eq:pd_def}
		p=\left\{
		\begin{aligned}
			&\geq 2, &&d=1, \\ 
			&> 2, && d = 2, \\
			&> \frac{12}{5}, &&d=3. 			
		\end{aligned}
		\right.  \tag{P}
	\end{equation}
	%
	Then we have the following two error estimates, one concerning the weakest regularity requirement on $V$ for the optimal convergence order, and the other concerning the highest convergence order that can be achieved under the assumption of $L^p+L^\infty$-potential for all $2 \leq p < \infty$. 
	
	\begin{theorem}[Optimal convergence]\label{thm:optimal}
		Assume that $V \in L^{p}(\R^d) + L^\infty(\R^d)$ with $p$ satisfying \cref{eq:pd_def}, and $\psi_0 \in H^2(\R^d)$. There exists $\tau_0>0$ sufficiently small such that when $0<\tau\leq\tau_0$, we have
		\begin{equation*}
			\| \psi(t_n) - \psi^n_\tau \|_{L^2} \lesssim \tau, \qquad 0 \leq n \leq T/\tau.   
		\end{equation*}
	\end{theorem}
	{According to the sharp $H^2$ well-posedness results recalled above, such first-order convergence in the $L^2$-norm is optimal in terms of the $H^2$-regularity of the exact solution. In addition, even with smoother initial data, the convergence order in the $L^2$-norm cannot exceed first-order for any standard exponential integrators in general due to the sharp $H^2$-regularity of the exact solution.} 
	
	\begin{theorem}[Convergence for $L^p$-potential]\label{thm:L2poten}
		Assume that $V \in L^p(\R^d) + L^\infty(\R^d)$, and $\psi_0 \in H^2(\R^d)$. There exists $\tau_0>0$ sufficiently small such that when $0<\tau\leq\tau_0$, we have, for $0 \leq n \leq T/\tau$, 
		\begin{equation*}
			\| \psi(t_n) - \psi^n_\tau \|_{L^2} \lesssim \left\{
			\begin{aligned}
				&\tau, &&d=1, \ p = 2,  \ \beta \in \R, \\
				&\tau^{1^-}, &&d=2, \ p = 2,  \ \beta \in \R, \\
				&\tau^{\frac{3}{4}^-}, &&d=3, \ p = 2, \ \beta = 0, \\ 
				&\tau^{{\frac{3}{4}}^{-} + 3(\frac{1}{2} - \frac{1}{p})}, &&d=3, \ 2 < p \leq \frac{12}{5}, \ \beta \in \R. 
			\end{aligned}
			\right. 
		\end{equation*}
	\end{theorem}
	
	As an immediate corollary, we have the following results for the inverse power potential
	\begin{equation}\label{eq:V}
		V(\vx) = \sum_{j=1}^J \frac{Z_j}{|\vx - \vx_j|^\alpha}, \qquad \vx \in \R^d, \  J \in \Z^+, \ \vx_j \in \R^d, \  Z_j \in \R, \  \alpha > 0. 
	\end{equation}
	First, for the most important case of the Coulomb potential in 3D, i.e. \cref{eq:V} with $\alpha = 1$ and $d=3$, we have the following optimal convergence result.  
	\begin{corollary}[Convergence for Coulomb potential]\label{cor:coulomb}
		Let $V$ be given by \cref{eq:V} with $\alpha = 1$ and $d=3$. Assuming that $\psi_0 \in H^2(\R^3)$, we have,  when $0<\tau\leq\tau_0$ with $\tau_0$ sufficiently small, 
		\begin{equation*}
			\| \psi(t_n) - \psi^n_\tau \|_{L^2} \lesssim \tau, \quad 0 \leq n \leq T/\tau. 
		\end{equation*}
	\end{corollary}
	For general forms of the inverse power potential, we have the following. 
	\begin{corollary}\label{cor:inverse_power}
		Let $V$ be given by \cref{eq:V}. Assuming that $\psi_0 \in H^2(\R^d)$, we have, when $0<\tau\leq\tau_0$ with $\tau_0$ sufficiently small, 
		\begin{equation*}
			\| \psi(t_n) - \psi^n_\tau \|_{L^2} \lesssim \left\{
			\begin{aligned}
				&\tau, &&d=1, \  0< \alpha<\frac{1}{2}, \\
				&\tau, &&d=2, \  0< \alpha<1, \\
				&\tau, &&d=3, \  0< \alpha<\frac{5}{4}, \\
				&\tau^{1^-}, && d=3, \  \alpha = \frac{5}{4}, \\
				&\tau^{\frac{3}{4}^- + (\frac{3}{2} - \alpha)}, &&d=3, \ \frac{5}{4} < \alpha < \frac{3}{2}, 
			\end{aligned}
			\right. \quad 0 \leq n \leq T/\tau. 
		\end{equation*}
	\end{corollary}
	
	\begin{remark}[$H^1$-norm convergence]
		Due to the filter function $\Pi_\tau$ in the EWI \cref{eq:EWI_scheme}, the $H^1$-norm convergence follows immediately from the $L^2$-norm convergence with \cref{lem:Pi} and the $H^2$-regularity of the exact solution $\psi$. Hence, in all the cases, the $H^1$-norm convergence order is half-order lower than the corresponding $L^2$-norm convergence order. We omit the details for brevity. 
	\end{remark}
	
	
	\section{Technical tools and estimates}\label{sec:3}
	\subsection{Estimates for the nonlinearity}
	In the following, to present the proof for \cref{thm:optimal,thm:L2poten} in a unified framework, we assume that, for some $2 \leq p_0 < \infty$,
	\begin{equation}\label{eq:assump_V}
		V \in L^{p_0}(\R^d) + L^\infty(\R^d).    
	\end{equation}
	Then there exist $ V_1 \in L^{p_0}(\R^d)$  and $V_2 \in L^\infty(\R^d)$ such that $V(\vx) = V_1(\vx) + V_2(\vx)$. Define, for $\phi: \R^d \rightarrow \C$, 
	\begin{equation}\label{eq:g_def}
		g(\phi) = g_1(\phi) + g_2(\phi), 
	\end{equation}
	where
	\begin{equation}\label{eq:g12_def}
		g_1(\phi):=V_1 \phi, \quad g_2(\phi):=(V_2 + \beta|\phi|^{2\sigma})\phi. 
	\end{equation}
	For $g_1$ and $g_2$, we have the following estimates. 
	\begin{lemma}\label{lem:g1}
		Let $V_1 \in L^{p_0}(\R^d)$ with $2 \leq p_0 < \infty$. Set 
		\begin{equation}\label{eq:r0}
			r_0 := \frac{2p_0}{p_0-1}, \qquad \tilde r_0 := \frac{p_0}{p_0-2} \ (\tilde r_0 := \infty \text{ if } p_0=2). 
		\end{equation}
		For $g_1$ in \cref{eq:g12_def}, we have
		\begin{align}
			&\| g_1(v) - g_1(w) \|_{L^{r_0'}} \leq \| V_1 \|_{L^{p_0}} \| v-w \|_{L^{r_0}},   \label{eq:g1-1}\\
			&\| g_1(v) - g_1(w) \|_{L^{\tilde r_0'}} \leq \| V_1 \|_{L^{p_0}} \| v-w \|_{L^{p_0}}.  \label{eq:g1-2}
		\end{align}
	\end{lemma}
	
	\begin{proof}
		By H\"older's inequality, we have, for $1 \leq \tilde r', r \leq \infty$, 
		\begin{equation}\label{eq:Vphi_general}
			\| V_1 \phi \|_{L^{\tilde r'}} \leq \| V_1 \|_{L^{p_0}} \| \phi \|_{L^{r}}, \quad \frac{1}{\tilde r} + \frac{1}{p_0} + \frac{1}{r} = 1. 
		\end{equation}
		Taking $\tilde r = r = r_0$ in \cref{eq:Vphi_general}, we obtain \cref{eq:g1-1}. Taking $\tilde r = \tilde r_0$ and $r = p_0$ in \cref{eq:Vphi_general}, we obtain \cref{eq:g1-2}.  
	\end{proof}
	
	\begin{lemma}\label{lem:g2}
		Let $V_2 \in L^\infty(\R^d)$, $\beta \in \R$, and $\sigma \in \R^+$. For $g_2$ in \cref{eq:g12_def}, we have
		\begin{equation*}
			\| g_2(v) - g_2(w) \|_{L^2} \lesssim  (\| V_2 \|_{L^\infty} + \| v \|_{L^\infty}^{2\sigma} + \| w \|_{L^\infty}^{2\sigma}) \| v-w \|_{L^2}. 
		\end{equation*}
	\end{lemma}
	The proof of \cref{lem:g2} can be found in \cite[Lemma 3.2]{bao2023_semi_smooth} and is thus omitted. 
	


	\subsection{Discrete Strichartz estimates}
	We recall the following properties of $\Pi_\tau$. 
	\begin{lemma}\label{lem:Pi}
		For any $1 \leq p < \infty$ and $\phi:\R^d \rightarrow \C$ measurable, we have
		\begin{align}
			&\| \Pi_\tau \phi \|_{L^p} \lesssim \| \phi \|_{L^p}, \label{eq:Pi-1}\\
			&\| \phi - \Pi_\tau \phi \|_{L^p} \lesssim \tau^\gamma \| (-\Delta)^{\gamma} \phi \|_{L^p}, \qquad \gamma \geq 0.   \label{eq:Pi-2}
		\end{align}
	\end{lemma}
	
	Let $\phi:\R^d \times \R \rightarrow \C$ be measurable. For any interval $I \subset \R$, we define the norms $L^q(I; L^r)$ with $1 \leq q, r \leq \infty$ as
	\begin{equation}
		\| \phi \|_{L^q(I; L^r)}:=
		\left\{
		\begin{aligned}
			&\left(\int_I \| \phi(\cdot, t) \|_{L^r}^q \rmd t\right)^\frac{1}{q}, &&1 \leq q < \infty, \\
			&\esssup_{t \in I} \| \phi(\cdot, t) \|_{L^r}, && q = \infty. 
		\end{aligned}
		\right. 
	\end{equation}
	At the time discrete level, for any $I \subset \R$ and $\tau > 0$, we define, for a sequence $(\phi^k(\vx))_{k \in \mathbb{Z}}$, 
	\begin{equation}\label{eq:norm_def}
		\| \phi \|_{l^q_\tau(I; L^r)} := 
		\left\{
		\begin{aligned}
			&\left(\tau \sum_{\tau k \in I} \| \phi^k \|_{L^r}^q\right)^\frac{1}{q}, &&1 \leq q < \infty, \\
			&\sup_{\tau k \in I} \| \phi^k \|_{L^r}, && q = \infty. 
		\end{aligned}
		\right.
	\end{equation}
	Define the filtered Schr\"odinger semigroup
	\begin{equation}
		\Stau(t) = e^{i t \Delta}\Pi_{\tau/4}, \quad t \in \R. 
	\end{equation}
	A pair $(q, r) \in [2, \infty] \times [2, \infty]$ is admissible if
	\begin{equation}\label{eq:admissible}
		\frac{2}{q} = d\left(\frac{1}{2} - \frac{1}{r} \right), \qquad (q, r, d) \neq (2, \infty, 2). 
	\end{equation}
	By Theorem 2.2 and Corollary 2.4 in \cite{choi2021}, and Theorem 4.2 in \cite{LRI_error_1} (also \cite{ignat2011} and \cite{choi2025}), we have the following discrete Strichartz estimates: 
	
		(i) Let $(q, r)$ be admissible, for all $\phi \in L^2(\R^d)$,  we have
		\begin{equation}\label{lem:dS1}
			\| S_\tau(\cdot) \phi \|_{l_\tau^q(\tau \Z; L^r)} \leq C(d, q) \| \phi \|_{L^2}. 
		\end{equation}
		
		(ii) Let $(q, r)$, $(\tilde q, \tilde r)$ be admissible and $(q, \tilde q) \neq (2, 2)$, for all $f \in l_\tau^{\tilde q'}(\tau \Z; L^{\tilde r'}(\R^d))$ and $\theta \in [-1, 1]$, we have
		\begin{equation}\label{lem:dS2}
			\left\| \tau \sum_{k = -\infty}^{n-1} S_\tau((n-k+\theta)\tau) f^k \right\|_{l_\tau^q(\tau \Z; L^r)} \leq C(d, q, \tilde q) \| f \|_{l_\tau^{\tilde q'}(\tau \Z; L^{\tilde r'})}. 
		\end{equation}
		
	(iii) Let $(q, r)$, $(\tilde q, \tilde r)$ be admissible and $(q, \tilde q) \neq (2, 2)$, for all $f \in L^{\tilde q'}(\tau \Z; L^{\tilde r'}(\R^d))$, we have
		\begin{equation}\label{lem:dSc}
			\left\| \int_{-\infty}^{n\tau} S_\tau(n\tau - s) f(s) \rmd s \right \|_{l_\tau^q(\tau \Z; L^r)} \leq C(d, q, \tilde{q}) \| f \|_{L^{\tilde q'}(\R; L^{\tilde r'})}. 
		\end{equation}

	With the use of \cref{lem:dSc}, we can get rid of establishing the discrete-time estimates in $l^q_\tau L^r$-norms for the exact solution $\psi$ to \cref{NLSE}. 
	\begin{remark}
		For any sequence $(g^k(\vx))_{k \in \Z}$ and $N_0 \geq 0$, by setting $f^k = g^k \mathbf{1}_{[0, N_0]}(k)$ in \cref{lem:dS2} with $\mathbf{1}_D$ being the indicator function of a set $D$, we have
		\begin{equation}\label{lem:dS2t}
			\left\| \tau \sum_{k = 0}^{n-1} S_\tau((n-k+\theta)\tau) g^k \right\|_{l_\tau^q([0, (N_0+1)\tau]; L^r)} \leq C(d, q, \tilde q) \| g \|_{l^{\tilde q'}_\tau([0, N_0\tau]; L^{\tilde r'})}\,.
		\end{equation}
		Similarly, for any $g:\R^d \times \C \rightarrow \C$ and $T_0 \geq 0$, by setting $f(s) = g(s) \mathbf{1}_{[0, T_0]}(s) $ in \cref{lem:dSc}, we have
		\begin{equation}\label{lem:dSct}
			\left\| \int_0^{n\tau} S_\tau(n\tau - s) g(s) \rmd s \right \|_{l_\tau^q([0, T_0]; L^r)} \leq C(d, q, \tilde{q}) \| g \|_{L^{\tilde q'}([0, T_0]; L^{\tilde r'})}\,.
		\end{equation}
	\end{remark}
	
	\section{Error estimate}\label{sec:4}
	In this section, we present the proof of the main results \cref{thm:optimal,thm:L2poten}. 
	
	We first highlight some particular admissible pairs that will be used in the proof. First, according to \cref{eq:admissible}, $(q, r) = (\infty, 2)$ is always admissible. Recalling \cref{eq:r0}, we have the following admissible pair. 
	\begin{lemma}\label{lem:pair1}
		For any $2 \leq p_0<\infty$ and $d=1,2,3$, by setting
		\begin{equation*}
			q_0 := \frac{4p_0}{d},  
		\end{equation*}
		we have $(q_0, r_0)$ is admissible and $q_0 \neq 2$. 
	\end{lemma}
	Then we define a constant 
	\begin{equation*}
		M:=\max_{(q, r) \in \mathcal{I}}\{\| \psi \|_{L^\infty([0, T]; H^2)}, \| \psi \|_{L^\infty[0, T]; L^\infty)}, \| \psi \|_{W^{1, q}([0, T]; L^{r})}\} < \infty, 
	\end{equation*}
	where $\mathcal{I} \supset \{(\infty, 2), (q_0, r_0)\}$ is a finite set of admissible pairs, which will be clear from the proof presented later. That $M < \infty$ follows from \cref{eq:H2-solution} and \cite[Theorem 4.8.1]{cazenave2003} which proves that for any admissible pair $(q, r)$, 
	\begin{equation*}
		\psi \in W_\text{loc}^{1, q}([0, T_\text{max}); L^r(\R^d)). 
	\end{equation*} 
	%
	%
	%
	
	Then we proceed to the error estimate of the EWI \cref{eq:EWI_scheme}. By the definition of $\Pi_\tau$ in \cref{eq:Pitau_def}, we immediately have that
	\begin{equation}\label{eq:Pitau_property}
		\Pi_\tau  = \Pi_{\tau/4}\Pi_\tau = \Pi_\tau \Pi_{\tau/4}. 
	\end{equation}
	Recalling \cref{eq:g_def} and iterating \cref{eq:EWI_scheme}, noting \cref{eq:Pitau_property}, we obtain
	\begin{align}\label{eq:EWI}
		\psi^n_\tau 
		&= e^{i n \tau \Delta} \psi^0_\tau -i \tau \sum_{k=0}^{n-1} e^{i(n-1-k)\tau\Delta} \vphi_1(i \tau \Delta) \Pi_\tau g(\psi^k_\tau) \notag \\
		&= \Stau(n\tau) \psi^0_\tau -i \tau \sum_{k=0}^{n-1} \Stau((n-1)\tau - k \tau) \vphi_1(i \tau \Delta) \Pi_\tau g(\psi^k_\tau), \quad 0 \leq n \leq T/\tau. 
	\end{align}
	By Duhamel's formula, we have
	\begin{equation}\label{eq:duhamel_ori}
		\psi(n\tau) = e^{i n \tau \Delta} \psi_0 - i \int_0^{n\tau} e^{i(n \tau - s)\Delta} g (\psi(s)) \rmd s, \quad 0 \leq n \leq T/\tau. 
	\end{equation}
	Applying $\Pi_\tau $ on both sides of \cref{eq:duhamel_ori}, recalling \cref{eq:Pitau_property}, we have
	\begin{equation}\label{eq:duhamel}
		\Pi_\tau \psi(n\tau) = \Stau(n\tau) \Pi_\tau \psi_0 - i \int_0^{n\tau} \Stau(n \tau - s) \Pi_\tau g (\psi(s)) \rmd s. 
	\end{equation}
	Subtracting \cref{eq:EWI} from \cref{eq:duhamel}, we obtain, for $0 \leq n \leq T/\tau$, 
	\begin{align}\label{eq:error_eq}
		\Pi_\tau \psi(n\tau) - \psi_\tau^n 
		&= \Stau(n\tau) \Pi_\tau \psi_0 - \Stau(n\tau) \psi_\tau^0 
		- i \int_0^{n\tau} \Stau(n \tau - s) \Pi_\tau g(\psi(s)) \rmd s \notag \\
		&\quad +i \tau \sum_{k=0}^{n-1} \Stau((n-1)\tau - k \tau) \vphi_1(i \tau \Delta) \Pi_\tau g(\psi_\tau^k) \notag \\
		&= \mathcal{E}_0^n + \mathcal{E}_1^n + \mathcal{E}_2^n + Z^n, 
	\end{align}
	where
	\begin{align}
		&\mathcal{E}_0^n := \Stau(n\tau) (\Pi_\tau \psi_0 - \psi^0_\tau), \label{eq:E0}\\
		& \mathcal{E}_1^n := - i \int_0^{n\tau} \Stau(n \tau - s) \Pi_\tau \left(g(\psi(s)) - g(\Pi_\tau \psi(s)) \right) \rmd s, \label{eq:E1}\\
		&\mathcal{E}_2^n := - i \int_0^{n\tau} \Stau(n \tau - s) \Pi_\tau g(\Pi_\tau \psi(s)) \rmd s \notag \\
		&\qquad\ \  + i \tau \sum_{k=0}^{n-1} \Stau((n-1)\tau - k \tau) \vphi_1(i \tau \Delta) \Pi_\tau g(\Pi_\tau\psi(k\tau)), \label{eq:E2}\\
		&Z^n := - i \tau \sum_{k=0}^{n-1} \Stau((n-1)\tau - k \tau) \vphi_1(i \tau \Delta) \Pi_\tau ( g(\Pi_\tau \psi(k\tau)) - g(\psi^k_\tau)). \label{eq:Z_def}
	\end{align}
	In the following, we estimate the four terms, respectively. From \cref{eq:E0}, direct application of \cref{lem:dS1} yields that for any admissible pair $(q, r)$ and interval $I \subset \R$, 
	\begin{equation}\label{eq:E0_est}
		\| \mathcal{E}_0 \|_{l^q_\tau(I; L^r)} \leq \| \mathcal{E}_0 \|_{l^q_\tau(\R; L^r)} \leq \widetilde{C} \| \Pi_\tau \psi_0 - \psi_\tau^0 \|_{L^2}, 
	\end{equation}
	where $\widetilde{C}$ is the constant in \cref{lem:dS1}. In fact, $\Pi_\tau \psi_0 - \psi_\tau^0 = 0$ in the above. We keep this term since it will be used in extending the local error estimates on some short interval to the global error estimate on $[0, T]$ (see \cref{eq:induction_I1}). 
	
	We start with $\mathcal{E}_2$ \cref{eq:E2} which is the local truncation error due to the time discretization in the EWI \cref{eq:EWI_scheme}. 
	\begin{proposition}\label{prop:E2}
		For any admissible pairs $(q, r)$ with $q \neq 2$, we have
		\begin{equation*}
			\| \mathcal{E}_2 \|_{l_\tau^q([0, T]; L^r)} \leq C(M, V, T) \tau. 
		\end{equation*}
	\end{proposition}
	
	\begin{proof}
		Note that
		\begin{equation}
			\tau \vphi_1(i \tau \Delta) = \int_0^\tau e^{i(\tau - s)\Delta} \rmd s, 
		\end{equation}
		which implies, by \cref{eq:Pitau_property}, 
		\begin{align}\label{sum-int}
			&\tau \Stau((n-1)\tau - k\tau) \vphi_1(i \tau \Delta) \Pi_\tau g(\Pi_\tau \psi(k \tau)) \notag \\
			&= \Stau((n-1)\tau - k\tau) \int_0^\tau e^{i(\tau - s)\Delta} \rmd s \Pi_\tau g(\Pi_\tau \psi(k \tau)) \notag \\
			&= \int_0^\tau \Stau(n\tau - k\tau - s) \rmd s  \Pi_\tau g(\Pi_\tau \psi(k \tau)) \notag\\
			&= \int_{k\tau}^{(k+1)\tau} \Stau(n\tau - s)  \Pi_\tau g(\Pi_\tau \psi(k \tau)) \rmd s .
		\end{align}
		Substituting \cref{sum-int} into \cref{eq:E2}, we obtain
		\begin{align*}
			\mathcal{E}^n_2
			&= -i \sum_{k=0}^{n-1} \int_{k\tau}^{(k+1)\tau} \Stau(n\tau - s) \Pi_\tau g(\Pi_\tau \psi(s)) \rmd s \\
			&\quad + i \sum_{k=0}^{n-1} \int_{k\tau}^{(k+1)\tau} \Stau(n\tau - s)  \Pi_\tau g(\Pi_\tau \psi(k \tau)) \rmd s \\
			&= -i \sum_{k=0}^{n-1} \int_{k\tau}^{(k+1)\tau} \Stau(n\tau - s) \Pi_\tau [g (\Pi_\tau \psi(s)) - g (\Pi_\tau \psi(k \tau))] \rmd s. 
		\end{align*}
		For $j=1, 2$, define, for $0 \leq k \leq n-1$,  
		\begin{equation}
			w_j(s) = \Pi_\tau [g_j(\Pi_\tau \psi(s)) - g_j(\Pi_\tau \psi(k \tau))], \quad  k\tau< s \leq (k+1)\tau, 
		\end{equation}
		and set $w_j(s) = 0 $ when $s > T$. 
		Then we have, for $1 \leq n \leq T/\tau$, 
		\begin{equation}
			\mathcal{E}_{2}^n = -i \sum_{k=0}^{n-1} \int_{k\tau}^{(k+1)\tau} \Stau(n\tau - s) (w_1(s) + w_2(s)) \rmd s = -i (\mathcal{E}_{2,1}^n + \mathcal{E}_{2,2}^n),  
		\end{equation}
		where
		\begin{equation}\label{E2}
			\mathcal{E}_{2,j}^n := -i \sum_{k=0}^{n-1} \int_{k\tau}^{(k+1)\tau} \Stau(n\tau - s) w_j(s) \rmd s = -i \int_0^{n\tau} S_\tau(n\tau - s) w_j(s) \rmd s. 
		\end{equation}
		For any admissible pairs $(q, r)$ and $(\tilde {q}, \tilde{r})$ with $(q, \tilde{q}) \neq (2, 2)$, by \cref{lem:dSct}, we have
		\begin{equation}\label{E2j}
			\| \mathcal{E}_{2, j} \|_{l^q_\tau([0, T]; L^r)} \lesssim \| w_j \|_{L^{\tilde q'}([0, T]; L^{\tilde r'})}, \quad j = 1, 2. 
		\end{equation}
		For $j=1$, choosing $(\tilde {q}, \tilde{r}) = (q_0, r_0)$ in \cref{E2j}, we have, with $N := T/\tau$, 
		\begin{align}\label{eq:w1_est}
			\displaybreak[2]
			&\| \mathcal{E}_{2, 1} \|_{l^q_\tau([0, T]; L^r)} \lesssim \| w_1 \|_{L^{q_0'}([0, T]; L^{r_0'})} = \left( \int_0^{T} \| w_1(s) \|_{L^{r_0'}}^{q_0'} \rmd s \right)^{\frac{1}{q_0'}} \notag \\
			&= \left( \sum_{k=0}^{N-1} \int_{k\tau}^{(k+1)\tau} \| \Pi_\tau [g_1(\Pi_\tau \psi(s)) - g_1(\Pi_\tau \psi(k \tau))] \|_{L^{r_0'}}^{q_0'} \rmd s \right)^{\frac{1}{q_0'}} \notag \\
			&\lesssim \| V_1 \|_{L^{p_0}} \left( \sum_{k=0}^{N-1} \int_{k\tau}^{(k+1)\tau} \| \psi(s) - \psi(k \tau) \|_{L^{ r_0}}^{q_0'} \rmd s \right)^{\frac{1}{q_0'}} \text{(\cref{eq:g1-1,eq:Pi-1})} \notag \\
			&= \| V_1 \|_{L^{p_0}} \left( \sum_{k=0}^{N-1} \int_{k\tau}^{(k+1)\tau} \left \| \int_{k\tau}^s \partial_t \psi(t) \rmd t \right \|_{L^{ r_0}}^{q_0'} \rmd s \right)^{\frac{1}{q_0'}} \notag \\
			&\leq \| V_1 \|_{L^{p_0}} \left( \sum_{k=0}^{N-1} \int_{k\tau}^{(k+1)\tau} \left ( \int_{k\tau}^s \| \partial_t \psi(t) \|_{L^{r_0}} \rmd t \right )^{q_0'} \rmd s \right)^{\frac{1}{q_0'}} \notag \\
			&\leq \| V_1 \|_{L^{p_0}} \left( \sum_{k=0}^{N-1} \int_{k\tau}^{(k+1)\tau} \int_{k\tau}^s \| \partial_t \psi(t) \|_{L^{r_0}}^{q_0'} \rmd t (s-k\tau)^{\frac{q_0'}{q_0}}  \rmd s \right)^{\frac{1}{q_0'}} \notag \\
			&\leq \tau^\frac{1}{q_0} \| V_1 \|_{L^{p_0}} \left( \sum_{k=0}^{N-1} \int_{k\tau}^{(k+1)\tau} \int_{k\tau}^s \| \partial_t \psi(t) \|_{L^{r_0}}^{q_0'} \rmd t \rmd s \right)^{\frac{1}{q_0'}} \notag \\
			&= \tau^\frac{1}{q_0} \| V_1 \|_{L^{p_0}} \left( \sum_{k=0}^{N-1} \int_{k\tau}^{(k+1)\tau} ((k+1)\tau-t) \| \partial_t \psi(t) \|_{L^{r_0}}^{q_0'} \rmd t \right)^{\frac{1}{q_0'}} \text{(change order)}\notag \\
			&\leq \tau^{\frac{1}{q_0} + \frac{1}{q_0'}} \| V_1 \|_{L^{p_0}} \left( \sum_{k=0}^{N-1} \int_{k\tau}^{(k+1)\tau} \| \partial_t \psi(t) \|_{L^{r_0}}^{q_0'} \rmd t \right)^{\frac{1}{q_0'}} \notag \\
			&=\tau  \| V_1 \|_{L^{p_0}} \| \partial_t \psi \|_{L^{q_0'}([0, T]; L^{r_0})} \leq \tau T^{1 - \frac{2}{q_0}} \| V_1 \|_{L^{p_0}} \| \partial_t \psi \|_{L^{q_0}([0, T]; L^{r_0})}. 
		\end{align}
		For $j=2$, choosing $(\tilde {q}, \tilde{r}) = (\infty, 2)$ in \cref{E2j}, following the same procedure with \cref{lem:g2}, we have
		\begin{align}\label{eq:w2_est}
			\displaybreak[2]
			\| \mathcal{E}_{2, 2} \|_{l^q_\tau([0, T]; L^r)} &\lesssim \| w_2 \|_{L^{1}([0, T]; L^2)} \notag \\
			&\leq \sum_{k=0}^{N-1} \int_{k\tau}^{(k+1)\tau} \| g_2( \Pi_\tau \psi(s)) - g_2(\Pi_\tau \psi(k \tau)) \|_{L^{2}} \rmd s \notag \\
			&\leq C(M, \| V_2 \|_{L^\infty}) \sum_{k=0}^{N-1} \int_{k\tau}^{(k+1)\tau} \| \psi(s) -  \psi(k \tau) \|_{L^{2}} \rmd s \notag \\
			&= C(M, \| V_2 \|_{L^\infty}) \sum_{k=0}^{N-1} \int_{k\tau}^{(k+1)\tau} \left \| \int_{k\tau}^s \partial_t \psi(t) \rmd t \right \|_{L^{2}} \rmd s \notag \\
			&\leq C(M, \| V_2 \|_{L^\infty}) \sum_{k=0}^{N-1} \int_{k\tau}^{(k+1)\tau} \int_{k\tau}^s \| \partial_t \psi(t) \|_{L^{2}} \rmd t \rmd s \notag \\
			&= C(M, \| V_2 \|_{L^\infty}) \sum_{k=0}^{N-1} \int_{k\tau}^{(k+1)\tau} ((k+1)\tau-t) \| \partial_t \psi(t) \|_{L^{2}} \rmd t \notag \\
			&\leq \tau C(M, \| V_2 \|_{L^\infty}) \| \partial_t \psi \|_{L^1([0, T]; L^{2})} \notag \\
			&\leq \tau T C(M, \| V_2 \|_{L^\infty}) \| \partial_t \psi \|_{L^\infty([0, T]; L^{2})}. 
		\end{align}
		The combination of \cref{eq:w1_est,eq:w2_est} yields from \cref{E2}
		\begin{equation*}
			\| \mathcal{E}_{2} \|_{l^q_\tau([0, T]; L^r)} \leq \| \mathcal{E}_{2, 1} \|_{l^q_\tau([0, T]; L^r)} + \| \mathcal{E}_{2, 2} \|_{l^q_\tau([0, T]; L^r)} \leq C(M, T, V) \tau, 
		\end{equation*}
		which completes the proof. 
	\end{proof}
	As is shown in the proof above, the local truncation error of the EWI essentially requires applying one time derivative of $g(\psi) = V \psi + \beta|\psi|^{2\sigma}\psi$, allowing for optimal first-order error in $L^2$-norm even though the potential is singular in space. 
	
	Then we estimate $\mathcal{E}^n_{1}$ \cref{eq:E1}, which is the error of the frequency truncation in the EWI \cref{eq:EWI_scheme}. We need an additional admissible pair. 
	\begin{lemma}\label{lem:pair2}
		For any $ 2 \leq p_0 \leq 4$, by setting
		\begin{equation*}
			\tilde q_0 := \frac{4p_0}{d(4-p_0)} \ (\tilde q_0 := \infty \text{ if } p_0 = 4), 
		\end{equation*}
		we have $(\tilde q_0, \tilde r_0)$ is admissible with $\tilde r_0$ given in \cref{eq:r0} and $\tilde q_0 \neq 2$. 
	\end{lemma}	
	
	\begin{proposition}\label{prop:E1}
		For any admissible pairs $(q, r)$ with $q \neq 2$,  we have
		\begin{equation*}
			\| \mathcal{E}_1 \|_{l^q_\tau([0, T]; L^r)} \leq C(M, V, T) \tau^\gamma , 
		\end{equation*}
		where
		\begin{equation}\label{eq:gamma}
			\gamma = \left\{
			\begin{aligned}
				&1, && p_0 \geq 2, \ d=1, \\
				&1, && p_0 > 2, \ d=2, \\
                &1^{-}, && p_0=2, \  d=2, \\
				&1, && p_0 > \frac{12}{5}, \ d=3, \\
				&{\frac{3}{4}}^{-} + 3\left( \frac{1}{2} - \frac{1}{p_0}  \right), && 2 \leq p_0 \leq \frac{12}{5}, \  d=3. 
			\end{aligned}
			\right.
		\end{equation}
	\end{proposition}
	
	\begin{proof}
		Recalling \cref{eq:E1,eq:g_def}, we have
		\begin{equation}\label{eq:E1_decom}
			\mathcal{E}^n_1 = \mathcal{E}^n_{1, 1} + \mathcal{E}^n_{1, 2},  
		\end{equation}
		where
		\begin{equation}
			\mathcal{E}^n_{1, j} = - i \int_0^{n\tau} \Stau(n \tau - s) \Pi_\tau(g_j(\psi(s)) - g_j(\Pi_\tau \psi(s))) \rmd s, \quad j=1, 2. 
		\end{equation}
		We start with the estimate of $\mathcal{E}^n_{1, 2}$. Using \cref{lem:dSct} with $(\tilde{q}, \tilde{r}) = (\infty, 2)$, by \cref{lem:Pi,lem:g2}, and H\"older's inequality, we have
		\begin{align}\label{eq:E12_est}
			\| \mathcal{E}_{1, 2} \|_{l^q_\tau([0, T]; L^r)} 
			&\lesssim \| \Pi_\tau(g_2(\psi) - g_2(\Pi_\tau \psi)) \|_{L^{1}([0, T]; L^{2})} \notag \\
			&\leq C(M, \| V_2 \|_{L^\infty}) \| \psi - \Pi_\tau \psi \|_{L^1([0, T]; L^2)} \notag \\
			&\leq C(M, \| V_2 \|_{L^\infty}) \tau \| \Delta \psi \|_{L^1([0, T]; L^2)} \notag \\
			&\leq C(M, \| V_2 \|_{L^\infty}) T\tau \| \psi \|_{L^\infty([0, T]; H^2)} \lesssim \tau. 
		\end{align}
		Then we consider $\mathcal{E}^n_{1, 1}$. We first present the estimates for optimal convergence, i.e. we consider $p_0$ such that $\gamma = 1$ in \cref{eq:gamma}. In fact, we can assume that $p_0 \leq 4$. Recalling \cref{lem:pair2}, applying \cref{lem:dSct} with $(\tilde q, \tilde r) = (\tilde q_0, \tilde r_0)$ to $\mathcal{E}^n_{1, 1}$, by \cref{eq:g1-2}, \cref{lem:Pi}, and H\"older's inequality, we have
		\begin{align}\label{eq:E11_est}
			\| \mathcal{E}_{1, 1} \|_{l^q_\tau([0, T]; L^r)}
			&\lesssim \| \Pi_\tau(g_1(\psi) - g_1(\Pi_\tau \psi)) \|_{L^{\tilde q_0'}([0, T]; L^{\tilde r_0'})} \notag \\
			&\leq \|V_1(I - \Pi_\tau) \psi \|_{L^{\tilde q_0'}([0, T]; L^{\tilde r_0'})} \notag \\
			&\leq \| V_1 \|_{L^{p_0}} \|  (I - \Pi_\tau)\psi \|_{L^{\tilde q_0'}([0, T]; L^{p_0})} \notag \\
			&\lesssim \tau T^{1-\frac{1}{\tilde q_0} - \frac{1}{q_1}} \| V_1 \|_{L^{p_0}} \| \Delta \psi \|_{L^{q_1}([0, T]; L^{p_0})}, 
		\end{align}
		where $q_1 = \frac{2p_0}{d(p_0 - 2)} \ (q_1 = \infty \text{ if } p_0 = 2)$ is such that $(q_1, p_0)$ is admissible. For $\| \Delta \psi \|_{L^{q_1}([0, T]; L^{p_0})}$, by the NLSE \cref{NLSE}, we have
		\begin{align}\label{eq:Deltapsiest}
			&\| \Delta \psi \|_{L^{q_1}([0, T]; L^{p_0})} \notag \\
			&\leq \| \partial_t \psi \|_{L^{q_1}([0, T]; L^{p_0})} + \| V_1 \psi \|_{L^{q_1}([0, T]; L^{p_0})} + \| (V_2 + \beta |\psi|^2)\psi \|_{L^{q_1}([0, T]; L^{p_0})} \notag \\
			&\leq M + \| V_1 \|_{L^{p_0}} \| \psi \|_{L^{q_1}([0, T]; L^\infty)} + \| V_2 + \beta |\psi|^2 \|_{L^\infty} \| \psi \|_{L^{q_1}([0, T]; L^{p_0})} \notag \\
			&\leq M + \| V_1 \|_{L^{p_0}} T^{\frac{1}{q_1}} \| \psi \|_{L^{\infty}([0, T]; L^\infty)} + C(M, \| V_2 \|_{L^\infty}) T^\frac{1}{q_1} \| \psi \|_{L^\infty([0, T]; H^2)} \notag \\
			&\leq C(M, T, \| V_1 \|_{L^{p_0}}, \| V_2 \|_{L^{\infty}}), 
		\end{align}
		which, plugged into \cref{eq:E11_est}, yields
		\begin{equation}\label{eq:E11_rst1}
			\| \mathcal{E}_{1, 1} \|_{l^q_\tau([0, T]; L^r)} \lesssim \tau. 
		\end{equation}
		
		Next, we estimate $\mathcal{E}^n_{1, 1}$ for $p_0$ such that $\gamma<1$ in \cref{eq:gamma}. There is no such $p_0$ in 1D and thus we only consider the 2D and 3D cases. Applying \cref{lem:dSct} with any admissible pair $(\tilde q, \tilde r) $ such that $\tilde q \neq 2$ to $\mathcal{E}^n_{1, 1}$, using \cref{eq:Vphi_general}, \cref{lem:Pi}, and H\"older's inequality, we have
		\begin{align}\label{eq:E11_est2}
			\| \mathcal{E}_{1, 1} \|_{l^q_\tau([0, T]; L^r)}
			&\lesssim \| \Pi_\tau(g_1(\psi) - g_1(\Pi_\tau \psi)) \|_{L^{\tilde q'}([0, T]; L^{\tilde r'})} \notag \\
			&\leq \|V_1(I - \Pi_\tau) \psi \|_{L^{\tilde q'}([0, T]; L^{\tilde r'})} \notag \\
			&\leq \| V_1 \|_{L^{p_0}} \|  (I - \Pi_\tau)\psi 	\|_{L^{\tilde q'}([0, T]; L^{r_1})} \notag \\
			&\lesssim \tau^\gamma \| V_1 \|_{L^{p_0}} \| (-\Delta)^\gamma \psi \|_{L^{\tilde q'}([0, T]; L^{r_1})}, 
		\end{align}
		where $1 \leq r_1 \leq \infty$ satisfies
		\begin{equation}\label{eq:relation}
			\frac{1}{\tilde r} + \frac{1}{p_0} + \frac{1}{r_1} = 1. 
		\end{equation}
		When $d=2$ and $p_0 = 2$, for any $\vep > 0$ sufficiently small, by choosing in \cref{eq:E11_est2}
		\begin{equation}
			\tilde r = \frac{1}{\vep}, \quad r_1 = \tilde q = \frac{2}{1-2\vep}, \quad \gamma = 1-\vep, 
		\end{equation}
		we have, by H\"older's inequality and Sobolev embedding $H^{2\vep}(\R^2) \hookrightarrow L^{r_1}(\R^2)$, 
		\begin{align}\label{eq:E11_rst2}
			\| \mathcal{E}_{1, 1} \|_{l^q_\tau([0, T]; L^r)} 
			&\lesssim \tau^{1-\vep} \| V_1 \|_{L^{p_0}} \| (-\Delta)^{1-\vep} \psi \|_{L^{\tilde q'}([0, T]; H^{2\vep})} \notag \\
			&\lesssim \tau^{1-\vep} \| (I-\Delta) \psi \|_{L^{\tilde q'}([0, T]; L^2)} \notag \\
			&\leq \tau^{1-\vep} T^{1-\frac{1}{\tilde q}}  \| \psi \|_{L^{\infty}([0, T]; H^{2})} \lesssim \tau^{1-\vep}. 
		\end{align}
		When $d=3$ and $2 \leq p_0 \leq 12/5$, for any $\vep > 0$ sufficiently small, choosing in \cref{eq:E11_est2}
		\begin{equation}
			\begin{aligned}
				&\tilde r = \frac{6}{1+4\vep}, \quad r_1 = \frac{6p_0}{(5-4\vep)p_0 - 6}, \quad \tilde q = \frac{2}{1-2\vep}, \\
				&\gamma = 1-s_1, \quad s_1 = \frac{3}{p_0} - \frac{5}{4} + \vep, \\
			\end{aligned}
		\end{equation}
		by the Sobolev embedding of Bessel potential spaces $ H^{2s_1, p_0}(\R^3)\hookrightarrow L^{r_1}(\R^3)$ \cite{cazenave2003}, and recalling \cref{eq:Deltapsiest}, we obtain
		\begin{align}\label{eq:E11_rst3}
			\| \mathcal{E}_{1, 1} \|_{l^q_\tau([0, T]; L^r)} 
			&\lesssim \tau^{1-s_1} \| V_1 \|_{L^{p_0}} \| (-\Delta)^{1-s_1} \psi \|_{L^{\tilde q'}([0, T]; H^{2s_1, p_0})} \notag \\
			&\lesssim \tau^{1-s_1} \| (I-\Delta) \psi \|_{L^{\tilde q'}([0, T]; L^{p_0})} \notag \\
			&\lesssim \tau^{1-s_1} T^{1 - \frac{1}{\tilde q} - \frac{1}{q_1}} \| (I-\Delta) \psi \|_{L^{q_1}([0, T]; L^{p_0})} \notag \\
			&\lesssim \tau^{1 - s_1} = \tau^{\frac{3}{4} + 3 (\frac{1}{2} - \frac{1}{p_0}) - \vep}.  
		\end{align}
		Combing \cref{eq:E12_est,eq:E11_rst1,eq:E11_rst2,eq:E11_rst3}, we obtain the desired result. 
	\end{proof}
	
	%
	
	Lastly, we estimate $Z$ in \cref{eq:Z_def}. By \cite[Lemma 11.1]{LRI_error_1}, we have the following. 
	\begin{lemma}\label{lem:phi_1_Pi}
		Let $0 < \tau \leq 1$. For any $ 1 \leq p \leq \infty$ and $f \in L^p(\R^d)$, we have 
		\begin{equation*}
			\| \vphi_1(i \tau \Delta) \Pi_\tau f \|_{L^p} \lesssim \| f \|_{L^p}. 
		\end{equation*}
	\end{lemma}
	
	
	\begin{proposition}\label{prop:Z}
		Let $I=[0, n\tau]$ for some $1 \leq n \leq T/\tau$. Assuming that $\sup_{0 \leq k \leq n-1} \| \psi^k_\tau \|_{L^\infty} \leq M_0$, we have, for any admissible pair $(q, r)$ with $q \neq 2$, 
		\begin{align*}
			\| Z \|_{l_\tau^q(I; L^r)} 
			&\leq (|I|^{1 - \frac{2}{q_0}} + |I|) (C_1 + C_2(M_0)) \| \Pi_\tau \psi - \psi_\tau \|_{X_\tau(I)}, 
		\end{align*}
		where $\| \cdot \|_{X_{\tau}(I)} := \| \cdot \|_{l_\tau^{q_0}(I; L^{r_0})} + \| \cdot \|_{l_\tau^\infty(I; L^2)}$, and $C_1$, $C_2$ also depend on $M$ and $V$. 
	\end{proposition}
	
	\begin{proof}
		Recalling \cref{eq:Z_def,eq:g_def}, we have
		\begin{equation*}
			Z^n = Z^n_1 + Z^n_2, 
		\end{equation*}
		where
		\begin{equation}\label{eq:Z12_def}
			Z_j^n = i \tau \sum_{k=0}^{n-1} \Stau((n-1)\tau - k \tau) \vphi_1(i \tau \Delta) \Pi_\tau ( g_j(\Pi_\tau \psi(k\tau)) - g_j(\psi^k_\tau)), \  j=1, 2. 
		\end{equation}
		For $Z_1$, by \cref{lem:dS2t} with $(\tilde q, \tilde r) = (q_0, r_0)$, \cref{lem:phi_1_Pi}, \cref{eq:g1-1}, and H\"older's inequality, we obtain
		\begin{align}\label{eq:Z1_est}
			\| Z_1 \|_{l^q_\tau(I; L^r)} 
			&\lesssim \| \vphi_1(i \tau \Delta) \Pi_\tau (g_1(\Pi_\tau \psi(k\tau)) - g_1(\psi^k_\tau)) \|_{l_\tau^{q_0'}([0, (n-1)\tau]; L^{ r_0'})} \notag \\
			&\lesssim \| V_1 \|_{L^{p_0}} \| \Pi_\tau \psi(k\tau) - \psi^k_\tau \|_{l_\tau^{q_0'}([0, (n-1)\tau]; L^{r_0})} \notag \\
			&\leq |I|^{1 - \frac{2}{q_0}} \| V_1 \|_{L^{p_0}} \| \Pi_\tau \psi(k\tau) - \psi^k_\tau \|_{l_\tau^{q_0}(I; L^{r_0})}. 
		\end{align}
		Similarly, by \cref{lem:dS2t} with $(\tilde q, \tilde r) = (\infty, 2)$, \cref{lem:phi_1_Pi}, \cref{lem:g2}, and H\"older's inequality, we have for $Z_2$, 
		\begin{align}\label{eq:Z2_est}
			\| Z_2 \|_{l^q_\tau(I; L^r)} 
			&\lesssim \| \vphi_1(i \tau \Delta) \Pi_\tau (g_2 (\Pi_\tau \psi(k\tau)) - g_2(\psi^k_\tau)) \|_{l_\tau^1([0, (n-1)\tau]; L^2)} \notag \\
			&\leq C(\| V_2 \|_{L^\infty}, M, M_0) \| \Pi_\tau \psi(k\tau) - \psi^k_\tau \|_{l_\tau^1([0, (n-1)\tau]; L^2)} \notag \\
			&\leq |I| C(\| V_2 \|_{L^\infty}, M, M_0)  \| \Pi_\tau \psi(k\tau) - \psi^k_\tau \|_{l^\infty_\tau(I; L^2)}. 
		\end{align}
		The combination of \cref{eq:Z1_est,eq:Z2_est} completes the proof. 
	\end{proof}
	
	
	\begin{proof}[Proof of \cref{thm:optimal,thm:L2poten}]
		Let $T_\ast<T$ be defined such that 
		\begin{equation}
			\left( |T_\ast|^{1 - \frac{2}{q_0}} + |T_\ast| \right)(C_1 + C_2(1+M)) \leq \frac{1}{4}, 
		\end{equation}
		where $C_1$ and $C_2$ are the constants in \cref{prop:Z}. We start by dividing the whole interval $[0, T]$ into $J$ subintervals of length no more than $T_\ast$, where $J \in \Z^+$ satisfies 
		\begin{equation}
			J (T_\ast/2) \leq T, \quad (J+1) (T_\ast/2) >T. 
		\end{equation}
		For $\tau \leq T_\ast/2$, choose $m_j \in \Z \ (j=0, \cdots, J+1)$ such that $m_0 = 0$, $m_{J+1} = T/\tau$, and 
		\begin{equation}
			m_j \tau \leq j (T_\ast/2), \quad (m_{j} + 1) \tau > j(T_\ast/2), \quad j = 1, \cdots, J.   
		\end{equation}
		Then we define time intervals $I_j \subset [0, T]$ as 
		\begin{equation}
			I_j = [m_{j} \tau, m_{j+1} \tau], \quad j = 0, \cdots, J. 
		\end{equation}
		It follows immediately that 
		\begin{equation}
			|I_j| = (m_{j+1} - m_{j})\tau \leq (T_\ast/2)+\tau \leq T_\ast, \quad j = 0, \cdots, J.  
		\end{equation}
		
		We shall present the proof for the cases $p \geq 2, d=1, 2$ and $p>2, d=3$, where we have $\gamma >3/4 \geq d/4$ in \cref{prop:E1}. When $p=2, d=3$, our result in \cref{thm:L2poten} only considers the linear case with $\beta = 0$, whose proof will follow the same line and is simpler as the induction process to control the nonlinearity is not needed. 
		
		Define the error function $e^n_\tau = \Pi_\tau \psi(n\tau) - \psi^n_\tau$ for $0 \leq n \leq T/\tau$. 
		We first show the error estimate on the time interval $I_0 = [0, m_1\tau]$.  We use mathematical induction to control the $L^\infty$-norm of $\psi^n_\tau$. First, we have
		\begin{equation}
			\| e^0_\tau \|_{L^2} = 0 \leq 2C_0 \tau^\gamma, \quad \| \psi_\tau^0 \|_{L^\infty} \leq 1+M. 
		\end{equation}
		We assume that, for $0 \leq m \leq n-1 \leq m_1\tau-1$, 
		\begin{equation}\label{eq:assumption}
			\| e^m_\tau \|_{L^2} \leq 2C_0 \tau^\gamma, \quad \| \psi_\tau^m \|_{L^\infty} \leq 1+M. 
		\end{equation}
		We shall prove \cref{eq:assumption} for $m = n$. Under the assumptions \cref{eq:assumption}, from \cref{eq:error_eq}, using \cref{eq:E0_est,prop:E1,prop:E2,prop:Z}, for $(q, r) \in \{(q_0, r_0), (\infty, 2)\}$, noting $n \tau \leq |I_0| \leq T_\ast$, we have
		\begin{align}\label{eq:error_est}
			&\| e_\tau \|_{l_\tau^q([0, n\tau]; L^r)} \notag \\
			&\leq \| \mathcal{E}_0 \|_{l_\tau^q([0, n\tau]; L^r)} + \| \mathcal{E}_1 \|_{l_\tau^q([0, n\tau]; L^r)} + \| \mathcal{E}_2 \|_{l_\tau^q([0, n\tau]; L^r)} + \| Z \|_{l_\tau^q([0, n\tau]; L^r)} \notag \\
			&\leq \widetilde{C}\| e^0_\tau \|_{L^2} + C_0 (\tau + \tau^\gamma) \notag \\
			&\quad +((n\tau)^{1 - \frac{2}{q_0}} + n \tau)(C_1 + C_2(1+M)) \| e_\tau \|_{X_\tau([0, n\tau])} \notag \\
			&\leq  \widetilde{C}\| e^0_\tau \|_{L^2} + C_0\tau^\gamma + \frac{1}{4} \| e_\tau \|_{X_\tau([0, n\tau])}, 
		\end{align}
		where $\widetilde{C}$ is the constant in \cref{lem:dS1} also used in \cref{eq:E0_est}, and $C_0$ is the sum of constants in \cref{prop:E1,prop:E2}. Taking $(q, r) = (q_0, r_0)$ and $ (q, r) = (\infty, 2)$ in \cref{eq:error_est}, and summing together, we have
		\begin{equation}\label{err1}
			\| e_\tau \|_{X_\tau([0, n\tau])} \leq 2\widetilde{C}\| e^0_\tau \|_{L^2} + 2C_0\tau^\gamma + \frac{1}{2} \| e_\tau \|_{X_\tau([0, n\tau])}, 
		\end{equation}
		which implies
		\begin{equation}\label{err2}
			\| e_\tau \|_{X_\tau([0, n\tau])} \leq 4\widetilde{C}\| e^0_\tau \|_{L^2} + 4C_0\tau^\gamma, 
		\end{equation}
		which inserted into \cref{eq:error_est} yields that
		\begin{equation}\label{err3}
			\| e_\tau \|_{l_\tau^\infty([0, n\tau]; L^2)} \leq \| e_\tau \|_{X([0, n\tau])} \leq 2\widetilde{C} \| e^0_\tau \|_{L^2} + 2C_0\tau^\gamma. 
		\end{equation}
		By (2.13) in \cite{choi2021}, we have, for any $\phi \in L^2(\R^d)$, 
		\begin{equation}\label{eq:inv}
			\| \Pi_{\tau/4} \phi \|_{L^\infty} \lesssim \tau^{-\frac{d}{4}} \| \phi \|_{L^2}. 
		\end{equation}
		Using \cref{eq:inv} with the observation $ e^n_\tau = \Pi_{\tau/4} e^n_\tau$, the Sobolev embedding $H^{\frac{d}{2}+\vep}(\R^d) \hookrightarrow L^\infty(\R^d)$ for some $\vep > 0$ sufficiently small, and \cref{eq:Pi-2}, when $\tau \leq \tau_0$ with $\tau_0 > 0$ sufficiently small depending on $M, V, T$, we have
		\begin{align}\label{eq:Linftybound}
			\| \psi_\tau^n \|_{L^\infty} 
			&\leq \| e^n_\tau \|_{L^\infty} + \| (I - \Pi_\tau) \psi(n \tau) \|_{L^\infty} + \| \psi(n \tau) \|_{L^\infty} \notag \\
			&\leq C\tau^{-\frac{d}{4}}\| e^n_\tau \|_{L^2} + C \| (I - \Pi_\tau) \psi(n \tau) \|_{H^{\frac{d}{2}+\vep}} + M \notag \\
			&\leq  2CC_0\tau^{\gamma-\frac{d}{4}} + C\tau^{1-\frac{d}{4} - \frac{\vep}{2}}\| \psi(n\tau) \|_{H^2} + M \notag \\
			&\leq 1 + M. 
		\end{align}
		Then we have proved \cref{eq:assumption} for $m = n$ and thus for all $0 \leq m \leq m_1$. As a result, we obtain the error estimate on the interval $I_0$ as  
		\begin{equation}\label{eq:error_est_I0}
			\| e_\tau \|_{X(I_0)} \leq 2C_0 \tau^\gamma, \quad \| \psi_\tau \|_{l_\tau^\infty(I_0; L^\infty)} \leq 1+M.  
		\end{equation}
		In particular, \cref{eq:error_est_I0} yields
		\begin{equation}\label{err-m1}
			\| e^{m_1}_\tau \|_{L^2} \leq 2C_0\tau^\gamma. 
		\end{equation}
		Then applying the same induction process on $I_1$, we have, for all $0 \leq m \leq m_2-m_1$,  
		\begin{equation}\label{eq:induction_I1}
			\| e_\tau \|_{X_\tau([m_1\tau, m_1\tau+m\tau])} \leq 2\widetilde{C}\| e^{m_1}_\tau \|_{L^2} + 2C_0\tau^\gamma, \quad \| \psi^{m_1 + m} \|_{L^\infty} \leq 1+M, 
		\end{equation}
		which implies, by noting \cref{err-m1}, 
		\begin{equation}\label{eq:err_I1}
			\| e_\tau \|_{X_\tau(I_1)} \leq 2C_0 \tau^\gamma (1+(2\widetilde{C})), \quad \| \psi_\tau \|_{l_\tau^\infty(I_1; L^\infty)} \leq 1+M. 
		\end{equation}
		Continuing this process on $I_2, \cdots, I_J$, we  can obtain, for any $0 \leq j \leq J$,  
		\begin{equation}\label{eq:err_Ij}
			\| e_\tau \|_{X(I_j)} \leq  2C_0 \tau^\gamma \sum_{l=0}^j(2\widetilde{C})^{l}, \quad \| \psi_\tau \|_{l_\tau^\infty(I_j; L^\infty)} \leq 1+M, 
		\end{equation}
		which concludes the proof. 

	\end{proof}
	
	\section{Numerical results}\label{sec:5}
	In this section, we present some numerical results to confirm our error estimates. We also apply the EWI to study the dynamics of the NLSE under multiple Coulomb potentials. To quantify the error, we define the following error functions:
	\begin{equation*}
		e_{L^2}(t_n) = \| \psi(t_n) - \psi^n_\tau \|_{L^2}, \quad e_{H^1}(t_n) = \| \psi(t_n) - \psi^n_\tau \|_{H^1}, \quad 0 \leq n \leq T/\tau. 
	\end{equation*}
	
	In order to do the numerical simulation, we have to truncate the whole space $\R^d$ into a bounded domain $\Omega := \Pi_{j=1}^d (a_j, b_j)$ which is large enough such that the wave function $\psi$ remains away from the boundary of $\Omega$, and equip $\Omega$ with periodic boundary conditions. Then a full discretization of the EWI \cref{eq:EWI_scheme} can be obtained by using the Fourier spectral method for spatial discretization, which is the same as adding a filter function $\Pi_\tau$ in the EWI-FS scheme in \cite{bao2023_EWI}. Note that since the singular potential $V \in L^2(\Omega)$, the Fourier spectral method is well-defined (while the Fourier pseudospectral method is not).  
	
	We first test the convergence of the EWI in 1D, 2D and 3D under two types of singular potentials: 

		(i) Inverse power potential: {for some $\alpha>0$ to be chosen such that the potential has expected regularity,} 
		\begin{equation}\label{eq:V1}
			V(\vx) = -\frac{1}{|\vx|^\alpha}, \qquad \vx \in \Omega. 
		\end{equation}
		
		(ii) $L^2$-potential generated in the Fourier space
		\begin{equation}\label{eq:V2}
			V(\vx) =\text{Re} \sum_{\bm{l} = (l_1, \cdots, l_d)^T \in \Z^d} \widehat v_{\bm l} e^{ i \bm{\mu_l} \cdot (\vx - \bm{a})}, \qquad \vx \in \Omega,  
		\end{equation}
		where $\widehat{v}_{\bm l} \in \C$ are the Fourier coefficients, $\bm a = (a_1, \cdots , a_d)^T$, and $\bm {\mu_l} \in \R^d$ with $(\bm{\mu_l})_j = 2 \pi l_j / (b_j - a_j)$.   

	The initial datum is chosen as the standard Gaussian
	\begin{equation}
		\psi_0(\vx) = e^{-\frac{|\vx|^2}{2}}, \qquad \vx \in \Omega. 
	\end{equation}
	{We remark here that although we choose a smooth initial datum as above, the same convergence behaviour can be observed for $H^2$-initial data if the potential only has $L^2$-regularity. In fact, as already mentioned before, even with a smooth initial data, the regularity of the exact solution is restricted by the regularity of the potential and cannot exceed $H^2$ in general if $V \in L_\text{loc}^2$ \cite{wu2025_singular1d,wu2025_singularhd}.} 
	
	For the 1D example, we set $\Omega = (-16, 16)$ and final time $T = 1$. In \cref{NLSE}, we choose  $\beta = 1$, $\sigma = 1$, and two inverse power potentials \cref{eq:V1} with $\alpha = 0.49$ ($L^2$) and $\alpha = 0.74$ ($L^{\frac{4}{3}}$), respectively. The reference solution is computed by the EWI with $\tau = \tau_\text{e} = 10^{-6}$ and $h = h_\text{e} = 2^{-9}$, When computing the errors, we fix the mesh size $h= h_\text{e}$ and vary $\tau$. 
	
	The numerical results are presented in \cref{fig:conv_dt_1D}, which show that under $L^2$-potential, the EWI is first-order convergent in $L^2$-norm and half-order convergent in $H^1$-norm. For the slightly more singular $L^\frac{4}{3}$-potential, there is convergence order reduction and the EWI is 0.8-order convergent in $L^2$-norm and $0.4$-order convergent in $H^1$-norm. These results confirm our error estimates and suggest that our regularity assumptions for optimal convergence are optimally weak. 
	
	\begin{figure}[htbp]
		\centering
		{\includegraphics[width=0.475\textwidth]{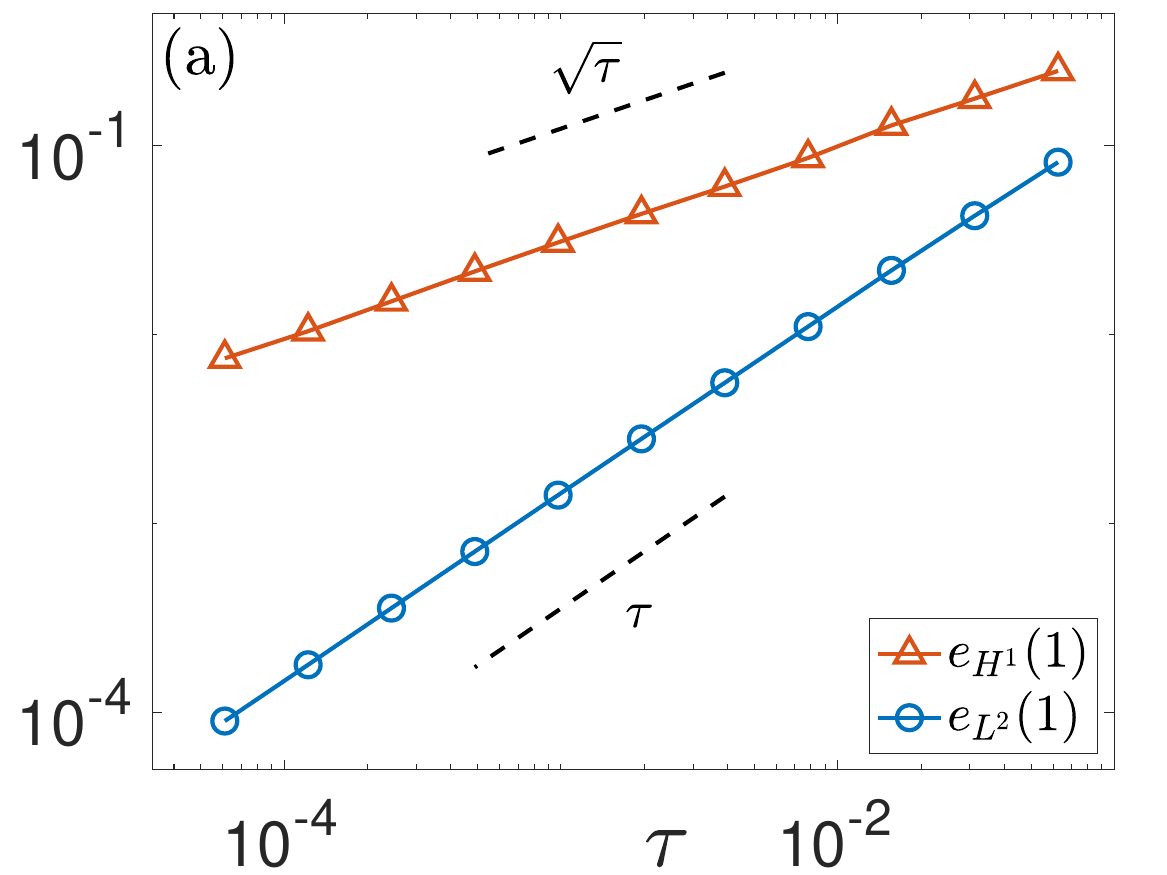}}\hspace{1em}
		{\includegraphics[width=0.475\textwidth]{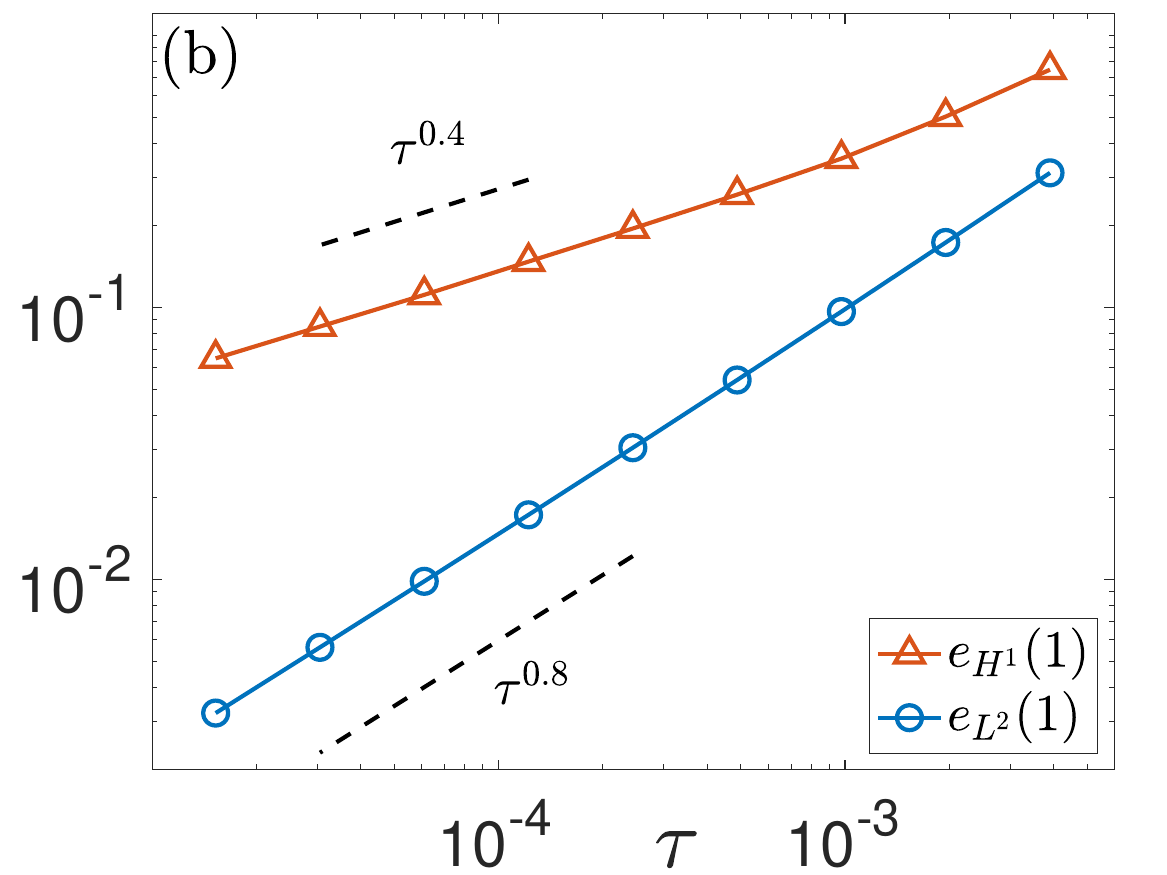}}
		\caption{Errors {at time $t=1$} in $L^2$- and $H^1$-norms of the EWI for the NLSE with singular potentials in 1D: (a) $\alpha = 0.49$ and (b) $\alpha = 0.74$}
		\label{fig:conv_dt_1D}
	\end{figure}
	
	{Although the spatial discretization is not considered in this paper, we present in \cref{fig:conv_h_1D} a numerical study on the spatial convergence of the Fourier spectral method for the above 1D example. We use the same reference solutions and compute the errors by fixing $\tau = \tau_\text{e}$ and varying $h$. We can observe that, under the $L^2$-potential, the Fourier spectral method is second-order convergent in $L^2$-norm and first-order convergent in $H^1$-norm. This is consistent with the $H^2$-regularity of the exact solution, indicating that the Fourier spectral method can still achieve optimal convergence orders with respect to the regularity of the exact solution even under a singular $L^2$-potential. The same behaviour of the Fourier spectral method are also observed in our extensive numerical tests of 2D and 3D examples with $L^2$-potentials (which are not shown here). For the more singular $L^\frac{4}{3}$-potential, the observed convergence rates reduce to about $1.5$-order in $L^2$-norm and $0.85$-order in $H^1$-norm. According to the sharp well-posedness result in \cite{wu2025_singular1d,wu2025_singularhd}, the exact solution in this case admits sharp $H^{1.75}$-regularity. This suggests that, once the potential becomes more singular than $L^2$, the Fourier spectral method may no longer achieve optimal convergence orders with respect to the regularity of the exact solution. A rigorous analysis of these phenomena is left for future work.}
	
	\begin{figure}[htbp]
		\centering
		{\includegraphics[width=0.475\textwidth]{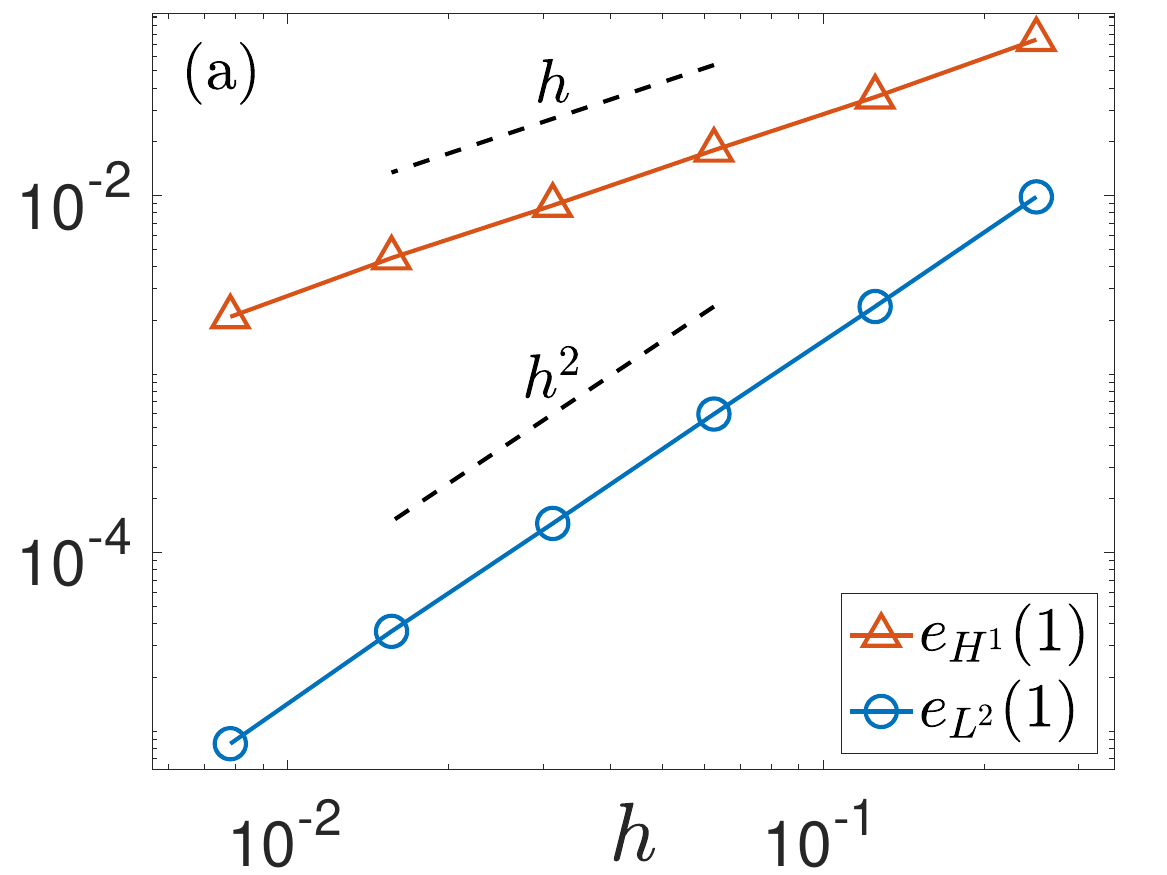}}\hspace{1em}
		{\includegraphics[width=0.475\textwidth]{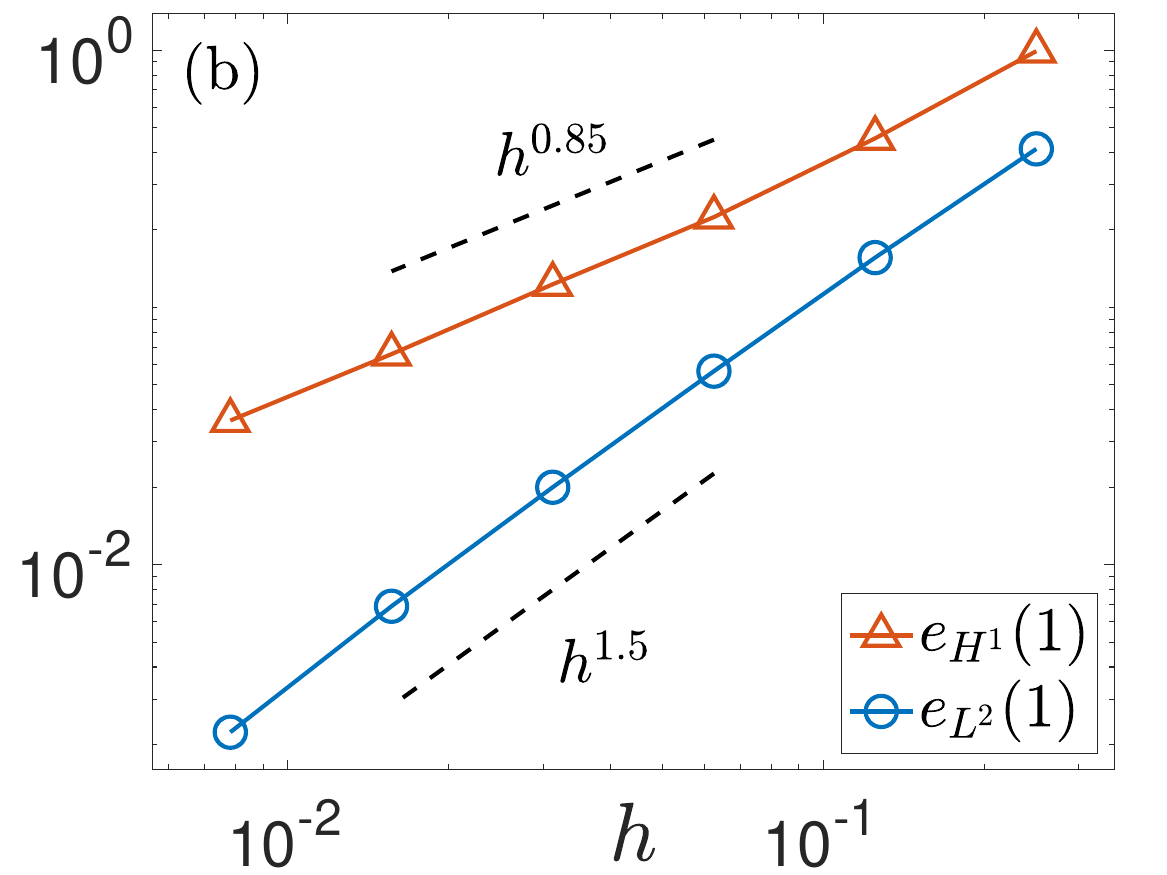}}
		\caption{{Spatial errors at time $t=1$ in $L^2$- and $H^1$-norms of the Fourier spectral method for the NLSE with singular potentials in 1D: (a) $\alpha = 0.49$ and (b) $\alpha = 0.74$}}
		\label{fig:conv_h_1D}
	\end{figure}
	
	Next, we consider a 2D example with $\Omega = (-8, 8) \times (-8, 8)$ and $T = 1/4$, and choose $\beta = -1$, $\sigma = 1$ in \cref{NLSE}. In this example, we study two kinds of singular potentials: (a) the Coulomb potential with $\alpha = 1$ in \cref{eq:V1}, which is in $L^{2-}$; (b) the random potential generated in the Fourier space through \cref{eq:V2} with 
	\begin{equation}
		\widehat{v}_{0, 0} = 1, \quad \widehat{v}_{\bm l} = \frac{\xi_{\bm l}}{|\bm{\mu_l}|}, \quad (0, 0)^T \neq \bm l \in \left\{-\frac{N_\text{ref}}{2}, \cdots, \frac{N_\text{ref}}{2}-1 \right\}^2, 
	\end{equation}
	where $\xi_{\bm l} =\text{rand}(-1, 1) + i \  \text{rand}(-1, 1)$ with $\text{rand}(-1, 1)$ returning a random number uniformly distributed in $[- 1, 1]$, and $N_\text{ref} = 2^9$. The reference solution is obtained by the EWI using $\tau_\text{e} = 10^{-5}$ and $h_\text{e} = 2^{-5}$ in both directions.  
	
	The numerical results are presented in \cref{fig:conv_dt_2D} (a) and (b) for the two potentials, respectively. In both cases, we see that the $L^2$-norm convergence is first-order and the $H^1$-norm convergence is half-order, corresponding well with our error estimates. We remark here that the convergence order reduction from $1$ to $1^{-}$ is hard to examine numerically. 
	\begin{figure}[htbp]
		\centering
		{\includegraphics[width=0.475\textwidth]{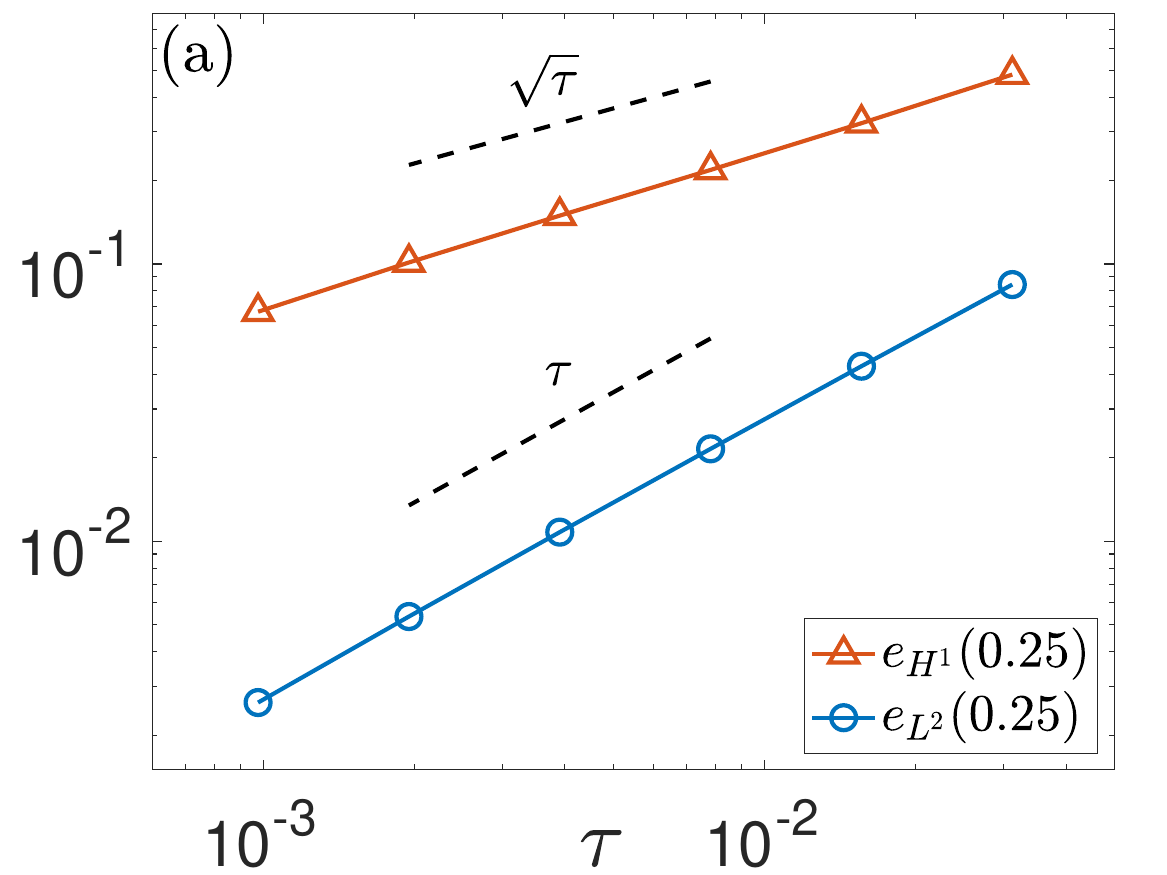}}\hspace{1em}
		{\includegraphics[width=0.475\textwidth]{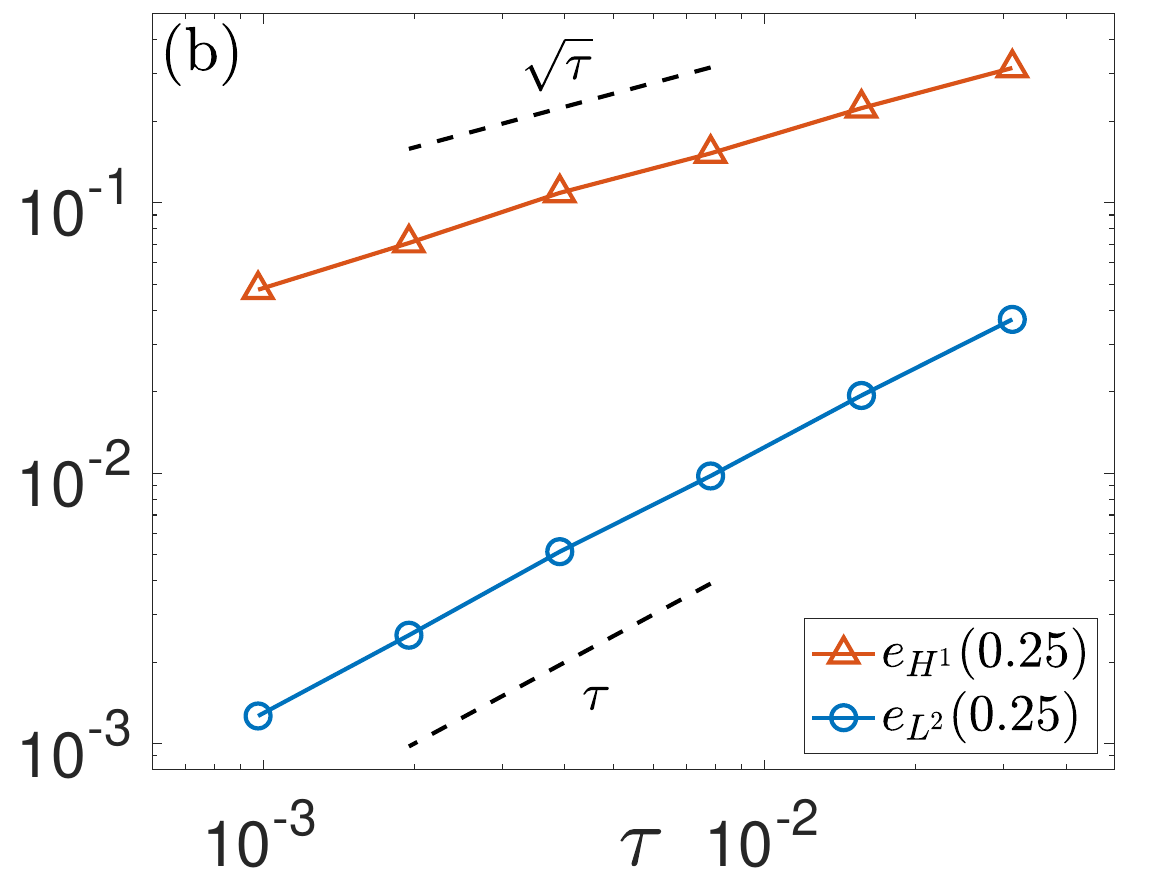}}
		\caption{Errors {at time $t=0.25$} in $L^2$- and $H^1$-norms of the EWI for the NLSE with $L^2$-potential in 2D: (a) Coulomb potential and (b) random potential}
		\label{fig:conv_dt_2D}
	\end{figure}
	
	Then we conduct a 3D numerical experiment with $\Omega = (-8, 8)^3$ and $T = 1/8$, and set $\beta = 1$, and $\sigma = 1$ in \cref{NLSE}. Two singular potentials of the form \cref{eq:V2} are considered: (a) an $L^{3^-}$-potential with $\widehat{v}_{\bm l} = (1+|\bm{\mu_l}|^2)^{-1}$ in \cref{eq:V2}, which has the same singularity as the Coulomb potential $-|\vx|^{-1}$ (recalling that the Fourier transform of $|\vx|^{-1}$ in $\R^3$ is $C |\vxi|^{-2}$); (b) an $L^2$-potential with $\widehat{v}_{\bm l} = (1+|\bm{\mu_l}|^2)^{-0.76}$ in \cref{eq:V2}. The reference solution is computed by the EWI using $\tau_\text{e} = 10^{-4}$ and $h_\text{e}= 2^{-5}$ in all the three dimensions. 
	
	We plot the numerical results in \cref{fig:conv_dt_3D}, where first-order convergence in $L^2$-norm is observed for both the $L^{3^-}$-potential in (a) and the $L^2$-potential in (b). These results verify our error estimates, while, the convergence order reduction for $L^2$-potential in 3D stated in \cref{thm:L2poten} is not observed, suggesting that first-order $L^2$-norm convergence might be extended to $L^p+L^\infty$-potential with $2 \leq p\leq \frac{12}{5}$. In addition, the $H^1$-norm convergence is 0.85-order for $L^{3^-}$-potential, which is due to the smooth Gaussian initial data and the better regularity of the $L^{3^-}$-potential (which is in fact $H^{\frac{1}{2}-}$). 
	\begin{figure}[htbp]
		\centering
		{\includegraphics[width=0.475\textwidth]{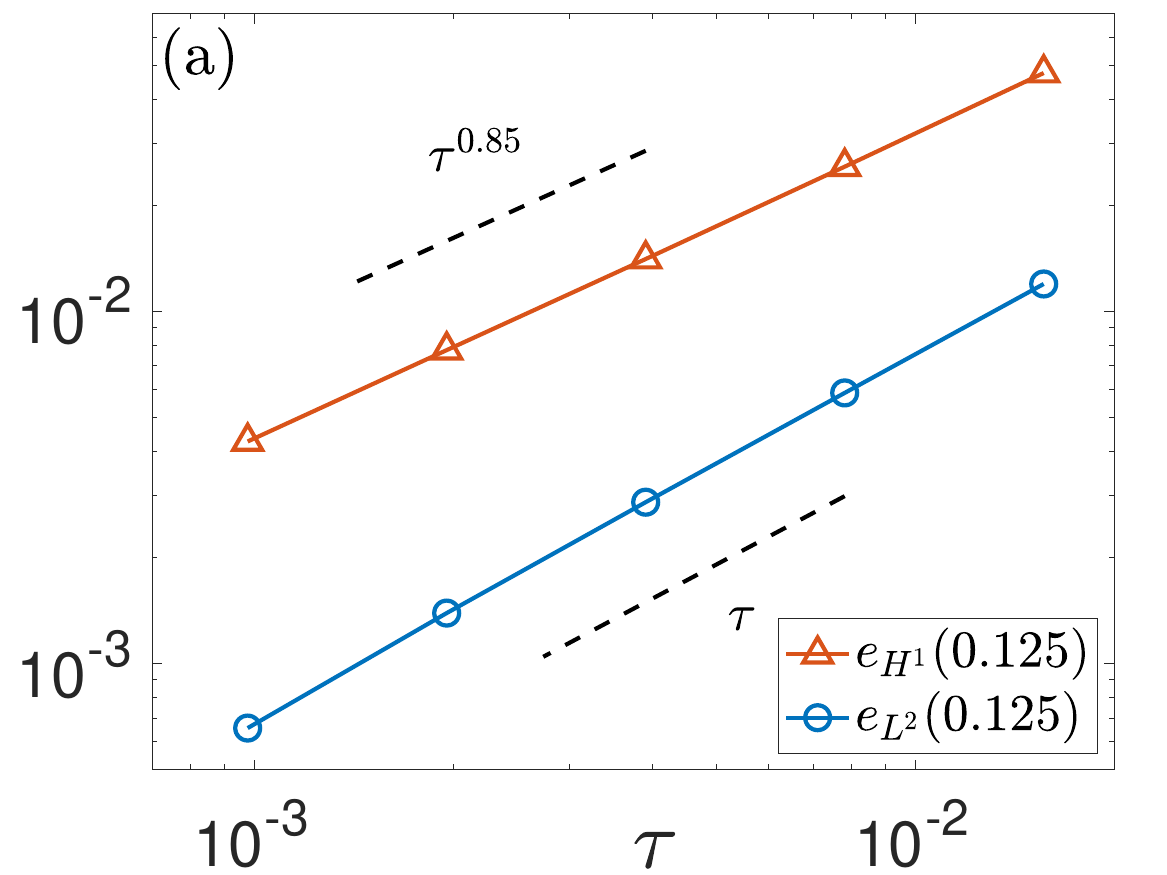}}\hspace{1em}
		{\includegraphics[width=0.475\textwidth]{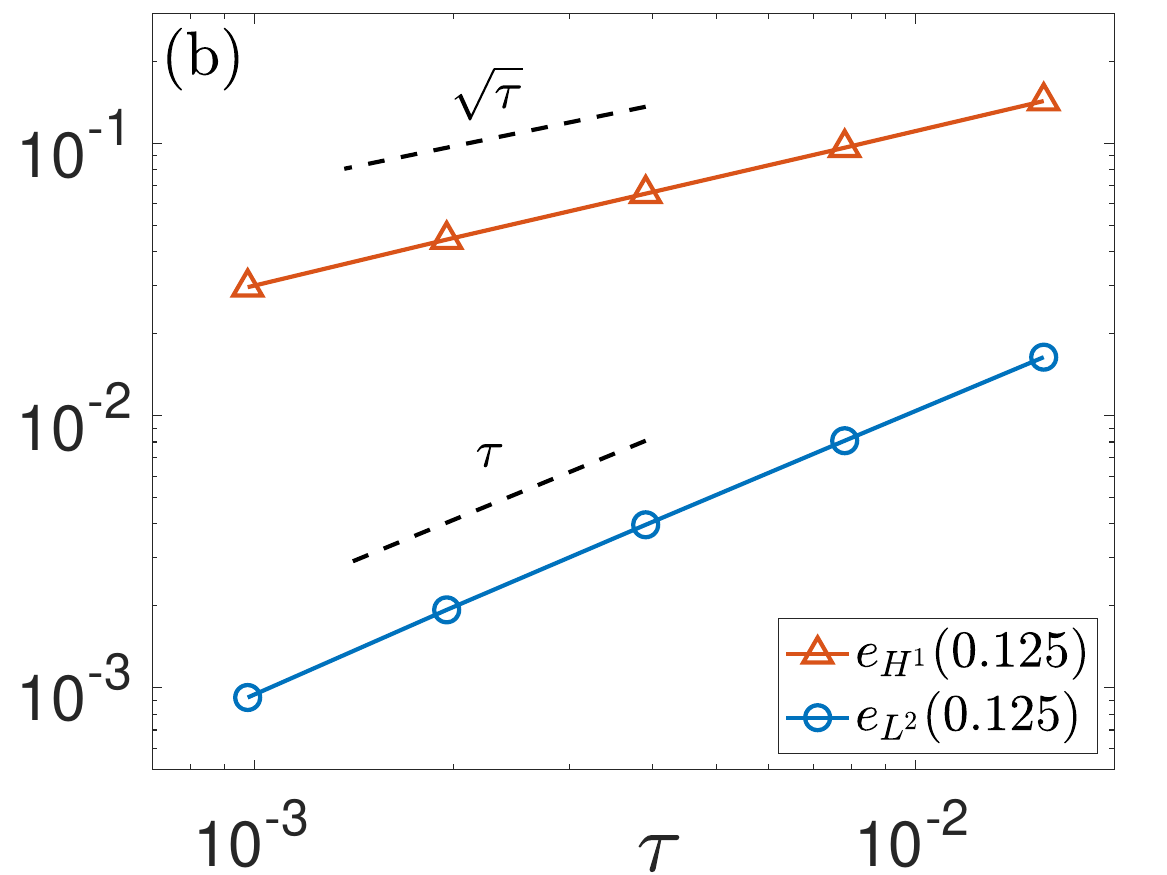}}
		\caption{Errors {at time $t=0.125$} in $L^2$- and $H^1$-norms of the EWI for the NLSE with singular potentials in 3D: (a) $L^{3^-}$-potential and (b) $L^2$-potential}
		\label{fig:conv_dt_3D}
	\end{figure}
	
	Finally, we apply the EWI \cref{eq:EWI_scheme} to simulate the dynamics under four attractive Coulomb potentials in 2D, which is given by \cref{eq:V} with $\alpha = 1$, $J=4$, $Z_j = -1$, and
	\begin{equation*}
		x_1 = (-1, 0)^T, \quad x_2 = (1, 0)^T, \quad x_3 = (0, 1)^T, \quad x_4 = (0, -1)^T. 
	\end{equation*}
	The initial datum is chosen as
	\begin{equation*}
		\psi_0(\vx) = \phi(\vx - \vx_0) e^{i x_1}, \quad \vx_0 = (-4, 2)^T, \quad \vx =(x_1, x_2)^T \in \Omega, 
	\end{equation*}
	where $\phi$ is the unique positive radial (action) ground state of the stationary NLSE 
	\begin{equation}
		-\Delta \phi(\vx) + \omega \phi(\vx) -|\phi(\vx)|^2 \phi(\vx) = 0, \quad \vx \in \R^2, \quad \omega = 3. 
	\end{equation}
	The initial set-up is illustrated in \cref{fig:ini_setting}. 
	\begin{figure}[htbp]
		\centering
		{\includegraphics[width=0.5\textwidth]{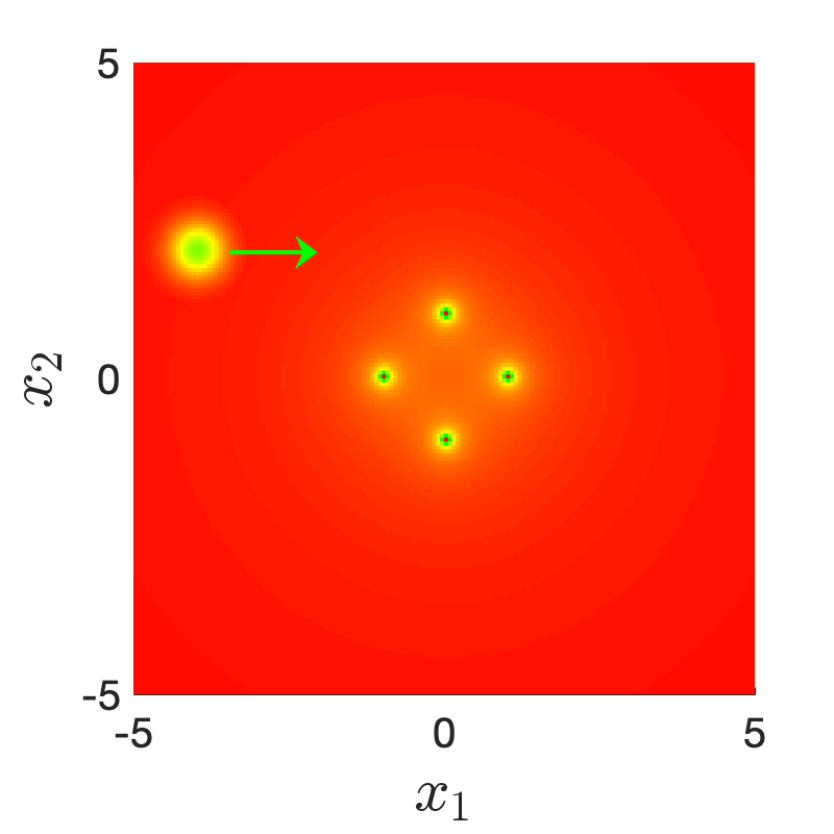}}
		\caption{Illustration of the initial set up}
		\label{fig:ini_setting}
	\end{figure}
	In computation, we choose $\Omega = (-8, 8) \times (-8, 8)$ and $T = 4$ with $\tau = 10^{-4}$ and $h = 2^{-5}$ for both directions. The density $|\psi(\vx, t)|^2$ at different time $t$ are exhibited in \cref{fig:dynamics}. 
	\begin{figure}[htbp]
		\centering
		{\includegraphics[width=\textwidth]{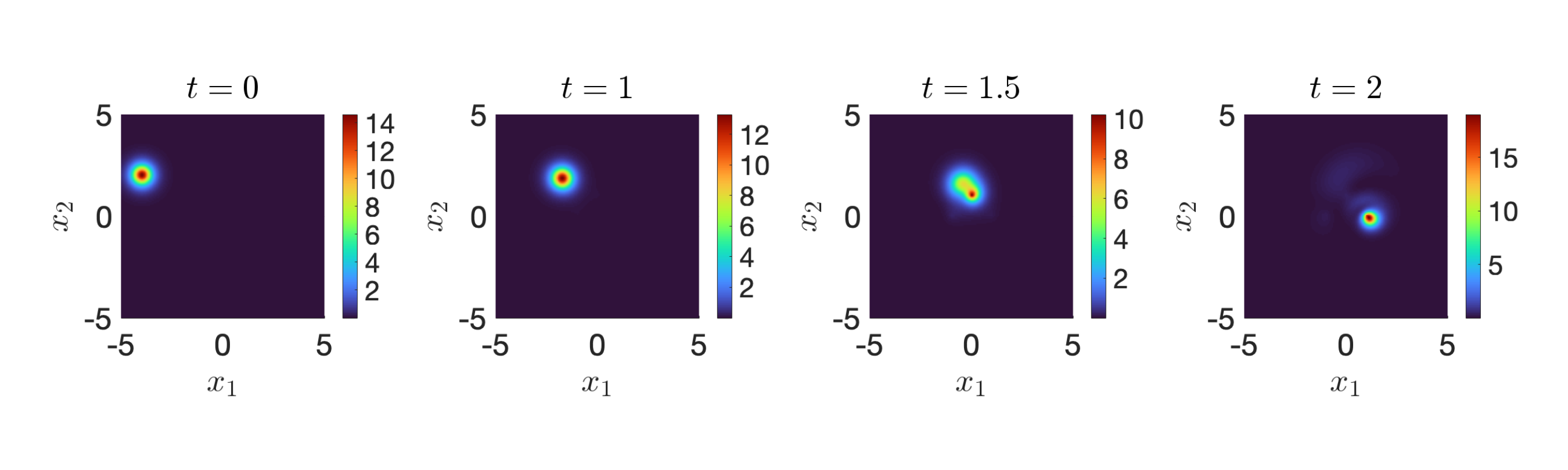}}\\
		{\includegraphics[width=\textwidth]{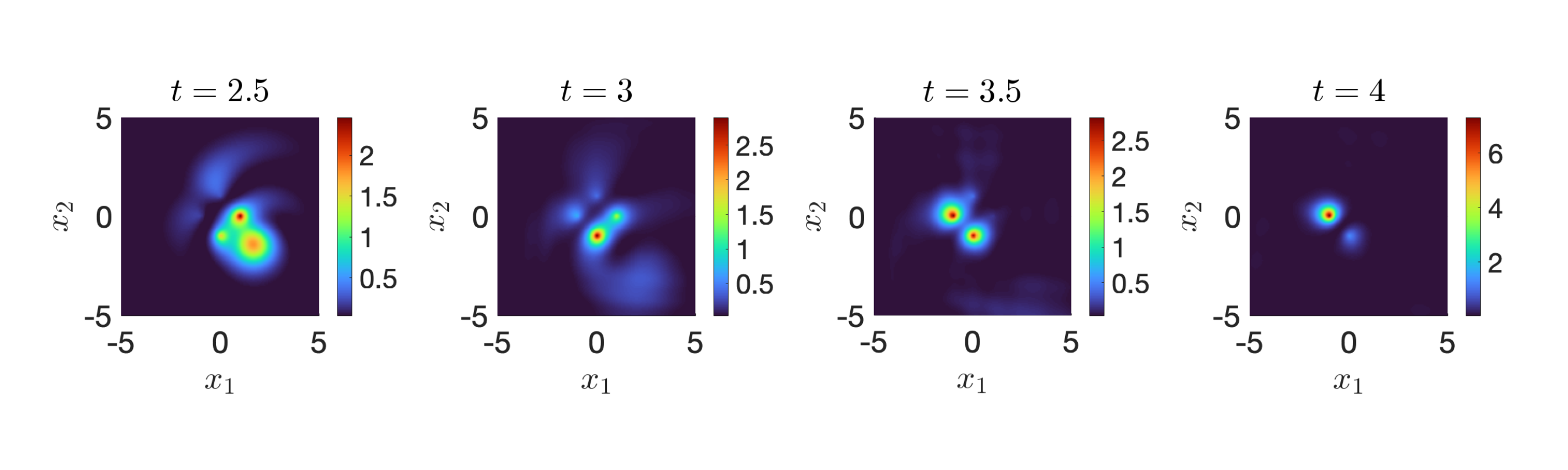}}
		\caption{Plots of $|\psi(\cdot, t)|^2$ at different time $t$ for the dynamics under four Coulomb potentials in 2D}
		\label{fig:dynamics}
	\end{figure}
	The numerical results demonstrate that the solitary wave initially moves to the right and becomes attracted to the upper Coulomb potential, subsequently visiting all four Coulomb centers in a clockwise sequence. 
	
	\section{Conclusion}\label{sec:6}
	We analyzed a first-order EWI for the NLSE with the singular potential of $L^p$-type under $H^2$-initial data. We obtained convergence orders for $L^p$-potential with $2 \leq p < \infty$ in 1D, 2D, and 3D. For $L^2$-potential, the EWI is almost first-order convergent in $L^2$-norm in 1D and 2D, while the convergence order is reduced to $\frac{3}{4}$-order in 3D. Under a stronger integrability assumption of $L^p$-potential with $p > \frac{12}{5}$ in 3D, the first-order $L^2$-norm convergence was proved. We also apply our results to the important cases of the inverse power potential. In particular, our results show that the EWI is optimally first-order convergent in $L^2$-norm under 3D Coulomb potential. Extensive numerical results in 1D, 2D, and 3D verified our error estimates and showed that they are optimal.  
	

\end{document}